\documentclass[11pt,a4paper,reqno]{amsart}
\usepackage[english]{babel}
\usepackage[T1]{fontenc}
\usepackage{verbatim}
\usepackage{palatino}
\usepackage{amsmath}
\usepackage{mathabx}
\usepackage{amssymb}
\usepackage{amsthm}
\usepackage{amsfonts}
\usepackage{graphicx}
\usepackage{esint}
\usepackage{color}
\usepackage{mathtools}
\usepackage{overpic}
\DeclarePairedDelimiter\ceil{\lceil}{\rceil}

\usepackage[colorlinks = true, citecolor = black]{hyperref}
\author{Amir Algom and Tuomas Orponen}
	
	\address{Department of Mathematics, University of Haifa at Oranim, Tivon 36006, Israel}

\email{\href{mailto:amir.algom@math.haifa.ac.il}{amir.algom@math.haifa.ac.il}}

\address{Department of Mathematics and Statistics\\ University of Jyv\"askyl\"a,
	P.O. Box 35 (MaD)\\
	FI-40014 University of Jyv\"askyl\"a\\
	Finland}

\email{\href{mailto:tuomas.t.orponen@jyu.fi}{tuomas.t.orponen@jyu.fi}}

\title{Uniformly perfect measures on strictly convex planar graphs are $L^{2}$-flattening}
\date{\today}
\thanks{A.A is supported by  the Israel Science Foundation (Grant No. 392/25),   NSF-BSF Grant No. 2024692, and  Grant No. 2022034 from the United States-Israel Binational Science Foundation (BSF), Jerusalem, Israel.\newline
T.O. is supported by the European Research Council (ERC) under the European Union’s Horizon Europe research and innovation programme (grant agreement No 101087499), and by the Research Council of Finland via the project \emph{Approximate incidence geometry}, grant no. 355453.}

\newcommand{\R}{\mathbb{R}}

\newcommand{\N}{\mathbb{N}}

\newcommand{\Z}{\mathbb{Z}}

\newcommand{\spt}{\operatorname{spt}}

\newcommand{\diam}{\operatorname{diam}}

\newcommand{\dist}{\operatorname{dist}}

\def\Barint_#1{\mathchoice
          {\mathop{\vrule width 6pt height 3 pt depth -2.5pt
                  \kern -8pt \intop}\nolimits_{#1}}%
          {\mathop{\vrule width 5pt height 3 pt depth -2.6pt
                  \kern -6pt \intop}\nolimits_{#1}}%
          {\mathop{\vrule width 5pt height 3 pt depth -2.6pt
                  \kern -6pt \intop}\nolimits_{#1}}%
          {\mathop{\vrule width 5pt height 3 pt depth -2.6pt
                  \kern -6pt \intop}\nolimits_{#1}}}

\numberwithin{equation}{section}

\theoremstyle{plain}
\newtheorem{thm}{Theorem}
\numberwithin{thm}{section}
\newtheorem*{"thm"}{"Theorem"}

\newtheorem{lemma}[thm]{Lemma}

\newtheorem{cor}[thm]{Corollary}
\newtheorem{proposition}[thm]{Proposition}
\newtheorem{"proposition"}[thm]{"Proposition"}
\newtheorem{"lemma"}[thm]{"Lemma"}

\newtheorem{claim}[thm]{Claim}

\theoremstyle{definition}

\newtheorem{definition}[thm]{Definition}

\theoremstyle{remark}
\newtheorem{remark}[thm]{Remark}

\newcommand{\nref}[1]{(\hyperref[#1]{#1})}

\DeclareMathSymbol{\intop}  {\mathop}{mathx}{"B3}

\begin{document}

\begin{abstract}
Uniformly perfect measures are a common generalisation of Ahlfors regular measures, self-conformal measures on the line, and their push-forwards under sufficiently regular maps. We show that every uniformly perfect measure $\sigma$ on a strictly convex planar $C^{2}$-graph is $L^{2}$-flattening. That is, for every $\epsilon>0$, there exists $p = p(\epsilon,\sigma) \geq 1$ such that
\begin{displaymath} \|\hat{\sigma}\|_{L^{p}(B(R))}^{p} \lesssim_{\epsilon,\sigma} R^{\epsilon}, \qquad R \geq 1. \end{displaymath}
\end{abstract}
\maketitle

\section{Introduction} This paper studies Fourier transforms of measures supported on planar graphs. Specifically, let $\varphi \in C^{2}(\R)$ be such that $\varphi''(x) > 0$ for $x \in [-2,2]$, so that $\varphi$ is strictly convex. Define 
$$\mathbb{P} := \mathbb{P}_\varphi  := \{(x,\varphi(x)) : x \in [-1,1]\},$$ 
the truncated graph of $\varphi$ over $[-1,1]$.  We fix such a function $\varphi$ throughout;  all implicit constants in the paper may depend on it. 

We are interested in the following  question. Suppose $\sigma$ is a Radon measure supported on $\mathbb{P}$. What can one say about the $L^{p}$-averaged growth rate of its Fourier transform? We will work with \emph{uniformly perfect} measures, a notion that first appeared in the work of Rossi and Shmerkin  \cite[equation (1.3)]{MR4163999}. Informally, uniformly perfect measures are quantitatively non-atomic at all scales and locations. Here is the precise definition.

\begin{definition}[$(D,\beta)$-uniformly perfect measure]\label{def:uniformlyPerfect1} Let $D > 1$ and $\beta \in [0,1)$. A locally finite Borel measure $\sigma$ on a metric space $(X,\rho)$ is called \emph{$(D,\beta)$-uniformly perfect} if $\diam(\spt \sigma) > 0$, and
\begin{displaymath} \sigma(B(x,r)) \leq \beta \cdot \sigma(B(x,Dr)) \end{displaymath}
for all open balls $B(x,r) \subset X$ such that $\spt \sigma \not\subset B(x,Dr)$. \end{definition} 
One example of uniformly perfect measures is given by Ahlfors $s$-regular measures with $s > 0$. Recall that a Borel measure $\nu$ on $\mathbb{R}^{d}$ is called Ahlfors $s$-regular if there exists a constant $C \geq 1$ such that
\begin{equation} \label{eq: ahlfors reg}
C^{-1}r^s \leq \nu \left( B(x,r) \right) \leq C r^s, \qquad x\in \spt \nu, \, 0 < r \leq \diam(\spt \nu).
\end{equation}
It is shown in \cite[Lemma 4.1]{MR4163999} that Ahlfors $s$-regular measures on $\mathbb{R}$ with $s > 0$ are uniformly perfect. Another important class of uniformly perfect measures are non-atomic self-conformal measures on the line. These are Borel probability measures on $\R$ satisfying the stationarity condition, for some strictly positive probability vector $\mathbf{p}=(p_1,...,p_n)$ and $\gamma>0$, \begin{equation} \label{eq: self-conformal}
\nu = \sum_{i=1} ^n p_i\cdot  f_i  \nu,\, \text{ where all  } f_i\in C^{1+\gamma}(\mathbb{R}^d) \text{ and } |f_i'|\in (0,1), \text{ where }  f_i\mu = \text{ pushforward of } \mu \text{ by } f_i.
\end{equation}
In \cite[Proposition 4.7 + Corollary 4.9]{MR4163999}, it is shown that such measures are uniformly perfect. In fact, \cite[Proposition 4.7]{MR4163999} deals with a much broader class of measures. Note that self-conformal measures may fail to be Ahlfors regular, see \cite{MR4072045} and references therein.

Finally, if a Radon measure on $[-1,1]$ is uniformly perfect, then so is its push-forward to $\mathbb{P}$ by $x \mapsto (x,\varphi(x))$. More generally, if $\sigma$ is $(D,\beta)$-uniformly perfect on $(X,\rho)$, and $T \colon (X,\rho) \to (Y,\rho')$ is a bilipschitz surjection, it is easy to check  that the push-forward $T\sigma$ is also uniformly perfect (see Lemma \ref{lemma4} for a similar argument).

Here is our main result.

\begin{thm} \label{thm:main} For every $D \geq 1$, $\mathfrak{d} > 0$, $\beta \in (0,1]$, and $\epsilon \in (0,1)$ there exists $p=p(D,\beta,\epsilon)\geq 1$ such that the following holds.

 Let $\sigma$ be a $(D,\beta)$-uniformly perfect probability measure with $\spt \sigma \subset \mathbb{P}$ and $\diam(\spt \sigma) \geq \mathfrak{d}$. Then,
\begin{equation}\label{eq:main thm flattening}
  \left\lVert \hat{\sigma}  \right\rVert_{L^p (B(R))}^{p} \lesssim_{D,\mathfrak{d},\beta,\epsilon} R^{\epsilon},\quad R\geq 1.
\end{equation}
\end{thm}
Here, $B(R)$ stands for the open $R$-ball centred at $0\in \mathbb{R}^2$, and $A \lesssim_{p} B$ means that $A \leq CB$, where $C > 0$ is a constant depending only on $p$.

Theorem \ref{thm:main} can be viewed as an $L^2$-flattening statement: upon mollifying the measure, iterated self-convolutions yield quantitative decay of its $L^2$ norm; see, for instance, Proposition \ref{prop:4.7} or Corollary \ref{cor2} below. For further discussion of the relation between \eqref{eq:main thm flattening} and flattening of (various notions of) $L^2$ norms, see \cite[Corollary 1.2]{algom2025khalil}.

\subsection{Previous work}\label{s:sota} Close relatives of Theorem \ref{thm:main} in previous literature are the $L^{2}$-flattening theorems of Rossi-Shmerkin \cite[Theorem 1.1]{MR4163999} on $\R$, and Khalil \cite[Theorem 1.6]{khalil2023exponential} on $\R^{d}$. These results are formulated in terminology different from Theorem \ref{thm:main}, but they are roughly the counterparts of Theorem \ref{thm:main} for (i) uniformly perfect measures on $\R$, and (ii) measures on $\R^{d}$ satisfying Khalil's \emph{uniformly affine non-concentration condition}. The reader should note that measures supported on smooth curves -- as in Theorem \ref{thm:main} -- do not satisfy Khalil's non-concentration condition, due to the presence of tangent lines; otherwise Theorem \ref{thm:main} could be deduced from the results in \cite{khalil2023exponential}. 

Thus, Theorem \ref{thm:main} and \cite[Theorem 1.6]{khalil2023exponential} are complementary, but finding a (natural) common generalisation seems like an interesting problem. For expert readers, we also mention that existing \emph{higher-dimensional inverse theorems} (by Hochman \cite{hochman2015self} and Shmerkin \cite{shmerkin2025inverse}) do not appear to be powerful enough to prove Theorem \ref{thm:main}, again due to the existence of tangent lines. Our proof will eventually rely on the one-dimensional inverse theorem of Shmerkin \cite[Theorem 2.1]{Sh}, see Section \ref{Section:sketch} for a brief explanation. 

We then explain the connection to a completely different strand of recent literature.  Recall that a measure $\mu$ on $\R^{d}$ is \emph{$s$-Frostman} if $\mu(B(x,r)) \lesssim r^{s}$ for $x \in \R^{d}$ and $r > 0$. Theorem \ref{thm:main} complements a sequence of recent papers \cite{Dasu2024demeter,MR4869897,Orponen2023add,Orponen2024Jan,orponen2025furstenberg} studying the $L^{p}$-averaged growth of Fourier transforms of $s$-Frostman measures supported on $\mathbb{P}$. Every $(D,\beta)$-uniformly perfect measure is $s$-Frostman for some $s = s(D,\beta) > 0$ (see Lemma \ref{frostmanLemma}) so these results also yield partial progress towards Theorem \ref{thm:main}. However, the major difference is that the sharp growth exponent in the variant of \eqref{eq:main thm flattening} for $s$-Frostman measures depends on $s$ (see \eqref{form60}), whereas it is independent of $D,\beta$ in Theorem \ref{thm:main} -- provided that $p$ is allowed to be arbitrarily large.

The state of the art in the analogue of Theorem \ref{thm:main} for $s$-Frostman measures is the following. Assume that $\varphi \in C^{3}(\R)$ with $\varphi'' > 0$, and $\sigma$ is an $s$-Frostman measure supported on $\mathbb{P}$. Then, for every $\epsilon > 0$ there exists $p = p(\epsilon,s) \geq 1$ such that
\begin{equation}\label{form60}  \|\hat{\sigma}\|_{L^{p}(B(R))}^{p} \lesssim R^{2 - \min\{3s,1 + s\} + \epsilon}, \qquad R \geq 1. \end{equation}
This was proven in \cite{orponen2025furstenberg}, and previously in \cite{Orponen2024Jan} in the case $\varphi(x) = x^{2}$. The exponent $\min\{3s,1 + s\}$ is sharp for $\varphi(x) = x^{2}$ (but sharpness remains open for general $\varphi$). The $C^{2}$-case (as in Theorem \ref{thm:main}) also remains open. Another intriguing problem is to determine if the exponent $p = 6$, or some other absolute constant, would suffice in \eqref{form60}. This was established by Yi \cite{yi2024bounded} for $s \geq 2/3$ (even for $\varphi \in C^{2}(\R)$), and earlier in \cite{Orponen2024Jan} for $\varphi(x) = x^{2}$. Finally, Demeter and Wang have shown that when $s \leq 1/2$, the estimate \eqref{form60} holds with $p = 6$, but with the non-sharp exponent $9s/4$ in place of $3s$.

The examples demonstrating the sharpness of \eqref{form60} are based on measures supported on (multi-scale) arithmetic progressions (see \cite[Example 1.8]{Orponen2023add}), and they are not relevant when $\sigma$ is Ahlfors regular -- or uniformly perfect, as Theorem \ref{thm:main} shows. 

Finally, Theorem \ref{thm:main} is related to several recent works studying  Fourier decay of stationary measures. The first author and Khalil \cite{algom2025khalil} proved  that the conclusion of Theorem \ref{thm:main} holds under the following assumptions. The measure $\sigma$ is the lift of a non-atomic self-similar measure on $\mathbb{R}$, onto  either (a) an analytic curve whose trace is not contained in an affine hyperplane of $\mathbb{R}^d$, or (b) a $C^{d + 1}$-curve $\gamma$ such that $\lbrace \gamma',\gamma'',...,\gamma^{(d)} \rbrace$ span $\mathbb{R}^d$ at every point. This was the first demonstration of a general non-trivial class of measures on curves for which $L^2$ flattening can be obtained, in the sense \eqref{eq:main thm flattening}. This work is also related to that of Algom, Chang, Meng Wu, and Yu-Liang Wu \cite{Algom2023Wu}, and the simultaneous independent work by Baker and Banaji \cite{baker2024polynomial}, on pointwise Fourier decay for smooth strictly convex push-forwards of self-similar measures.  It is also related to the subsequent paper of Baker, Khalil, and Sahlsten \cite{khalil2024polynomial}, as well as to  \cite{algom2020decay, algom2023polynomial, algom2024plane,Baker2023Sahl, banaji2025fourier,Dai2007Feng,MosqueraShmerkin,solomyak2019fourier, streck2023absolute, Tsujii2015self}. For more on the relation between Theorem \ref{thm:main} and the Fourier decay problem for stationary measures, see \cite[Section 1.2]{algom2025khalil}.

To compare Theorem \ref{thm:main} to  \cite[Theorem 1.1]{algom2025khalil}, the latter is valid for a more general class of curves, and in all dimensions, whereas Theorem \ref{thm:main} is stated for measures supported on $\mathbb{P}\subset \mathbb{R}^2$. On the other hand, self-similar measures (and their lifts to $\mathbb{P}$) are uniformly perfect, so Theorem \ref{thm:main} handles more general -- not necessarily stationary -- measures. It is plausible that Theorem \ref{thm:main} extends to more general curves in all dimensions, but we leave this for future research.

It is also natural to consider analogues of Theorem \ref{thm:main} for surfaces. For instance, one may ask whether  a Borel probability measure $\sigma$  on the paraboloid (or any other "curved" hypergraph) in $\mathbb{R}^{d+1}$ is $L^2$ flattening, provided its projection to $\mathbb{R}^d$ is uniformly affinely non-concentrated in the sense of Khalil \cite{khalil2023exponential}. More generally, it may be plausible that $\sigma$ is $L^2$ flattening whenever its projection to $\mathbb{R}^d$ is $L^2$ flattening. A strategy like this underlies the proof of
$L^2$ flattening for self-similar measures on curves, cf. \cite[Theorem 1.3]{algom2025khalil}.

\subsection{Proof outline} \label{Section:sketch}  The main step in the proof Theorem \ref{thm:main} is Lemma \ref{lemma1}.  We now state a slightly inaccurate version of (a weaker version of) that lemma, and outline its proof. Afterwards we briefly explain how Theorem \ref{thm:main} is deduced from the lemma. 

Recall that an Ahlfors regular set is the support of an Ahlfors regular measure, as in \eqref{eq: ahlfors reg}. For $\delta>0$ and $X \subset \R^{d}$, we denote by $|X|_{\delta}$  the  $\delta$-covering number of $X$.
\begin{"lemma"} \label{prop:toy}
Suppose $X,Y\subset [0,1]^2$ and $\delta \in 2^{-\mathbb{N}},N\in \mathbb{N}$, are such that:
\begin{enumerate}
\item For every $\sqrt{\delta}$-square $Q \subset \R^{2}$ intersecting $X$,
$$|Q \cap X|_{\delta} \approx N, \text{ independently of } Q.$$

\item $Y \subset \mathbb{P}$ is  an Ahlfors regular set. 

\item $|X + Y|_{\delta} \approx |X|_{\delta}$.
\end{enumerate}
Then, 
$$N \approx \delta^{-1}.$$
 \end{"lemma"}
The implicit constants behind the sloppy "$\approx$" notation depend on the Ahlfors regularity (exponent and constant) of $Y$, and the curvature of $\mathbb{P}$. We are liberal about these dependencies in the "Lemma" and its proof; the reader should consult  Lemma \ref{lemma1} for the full details. 

Our proof is based on Shmerkin's inverse theorem for $L^{q}$-norms \cite[Theorem 2.1]{Sh}, which in turn is inspired by Hochman's inverse theorem for entropy \cite[Theorem 2.7]{MR3224722}. In fact, we employ a corollary of Shmerkin's theorem obtained by Rossi and Shmerkin \cite[Proposition 3.1]{MR4163999}, which states that convolution with a uniformly perfect measure on the real line is $L^{2}$-flattening. One may  think of $Y$ in the "Lemma" as the support of our uniformly perfect measure $\sigma$ in Theorem \ref{thm:main}, although this is slightly misleading; Lemma \ref{lemma1} deals directly with the measure $\sigma$, not its support. The support analogy is more accurate if $\sigma$ happens to be Ahlfors regular (and Figure \ref{fig1} depicts this case).

\begin{figure}[h!]
\begin{center}
\begin{overpic}[scale = 0.8]{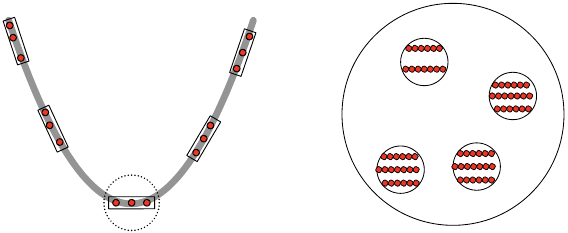}
\put(5,35){$\mathbb{P}$}
\put(21,11){\small{$B$}}
\put(72.5,21.5){$Q$}
\put(64,29){$X$}
\end{overpic}
\caption{Left: the set $Y \subset \mathbb{P}$ and the $\Delta$-disc $B$. Right: the structure of $X$ under the hypothesis $|X \cap Y_{B}|_{\delta} \approx |X|_{\delta}$.}\label{fig1}
\end{center}
\end{figure}
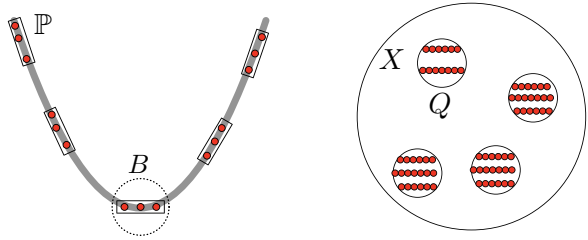

$$ $$
\begin{proof}[Proof sketch] Let $N$ be as in  condition (1). Evidently $N \lesssim \delta^{-1}$, so the claim $N \approx \delta^{-1}$ means that $N$ is nearly maximal. Write $\Delta := \sqrt{\delta}$, and let $B$ be a fixed $\Delta$-disc centred at $Y \subset \mathbb{P}$. There is no harm in visualising $B$ as centred at $0$, as in Figure \ref{fig1}. Since $\mathbb{P}$ is $C^{2}$, the intersection $Y_{B} := Y \cap B$ is "flat" in the sense that it is contained in a rectangle $R$ of dimensions $\delta \times \Delta$. Clearly $|X + Y_{B}| \lessapprox |X|$ by condition (3), which implies 
\begin{equation}\label{form59} |(X \cap Q) + Y_{B}| \lessapprox |X \cap Q| \end{equation}
for a "typical" square $Q \in \mathcal{D}_{\Delta}(X)$.

We now claim that \eqref{form59} forces $X \cap Q$ to have the structure shown on the right of Figure \ref{fig1}: $X \cap Q$ is organised into $m \in \N$ horizontal rows $R_{1},\ldots,R_{m}$ which are "full" in the sense
\begin{equation}\label{form61} |X \cap R_{j}|_{\delta} \approx \Delta^{-1}. \end{equation}
Here the  structure of $Y\subset \mathbb{P}$ (in particular, the $1$-dimensionality of $\mathbb{P}$) comes into play: 
thanks to the flatness of $\mathbb{P} \cap B$ at scale $\delta$, the $\delta$-neighbourhood of $Y_{B}$ coincides (up to translation) with the $\delta$-neighbourhood of an Ahlfors $s$-regular subset of $[0,\Delta] \subset \R$. One can then cover  $X\cap Q$ by rectangles $\lbrace R' \rbrace$ of dimensions $\delta \times \Delta$ such that:
\begin{enumerate}
\item Each $R'$  has the same orientation as $R$.

\item The intersection $R'\cap X \cap Q$ can  be  identified with a subset of $[0, \Delta] \subset \R$.

\item  For "most" $R'$,  \eqref{form59} remains true for $R'\cap X \cap Q$ in lieu of $X\cap Q$.
\end{enumerate}
 Thus, via \eqref{form59} we obtain a lack-of-growth type statement for corresponding sumsets on $\mathbb{R}$. Using the Ahlfors regularity of $Y_B$, we may  apply Rossi and Shmerkin's result  \cite[Theorem 2.1]{Sh}, or more precisely Theorem \ref{RSProp} in our case, for "most" intersections $R'\cap X \cap Q$. The conclusion is that the claimed structure of $X \cap Q$. The rigorous version of this argument is Claim \ref{c1}.

Let $m$ be the number or rows $R_j$ satisfying \eqref{form61}. Thus, so far we have shown that  
$$|X \cap Q|_\delta \gtrapprox m \Delta^{-1}.$$
To obtain $N \approx \delta^{-1}$, it remains to show that  
$$m \approx \Delta^{-1}.$$
This  is where the curvature of $\mathbb{P}$ is finally used. Namely, we pick a second $\Delta$-disc $B'$ centred at $Y$ with $\dist(B,B') \approx 1$, and run the previous argument again. This will show that $X \cap Q$ has a row structure, as described above, in two distinct and "transversal" directions. This forces $N \approx |X \cap Q|_\delta \approx \delta^{-1}$. \end{proof}

After proving the "Lemma" (which we recall is a toy version of Lemma \ref{lemma1}), we know that the conditions (1)-(3) in the "Lemma" above imply $|Q \cap X|_{\delta} \approx \delta^{-1}$ for a typical $Q \in \mathcal{D}_{\sqrt{\delta}}(X)$, therefore $|X|_\delta \gtrapprox \delta^{-1}$. The next step towards Theorem \ref{thm:main} is to upgrade this information to $|X|_\delta \approx \delta^{-2}$. This is based on the following "iteration". We re-apply the "Lemma" at scale $\Delta = \sqrt{\delta}$, drawing the conclusion that  "typical" squares $\mathbf{Q} \in \mathcal{D}_{\sqrt{\Delta}}(X)$ satisfy $$|\mathbf{Q} \cap X|_{\Delta} \approx \Delta^{-1}.$$ 
Combining this with $|Q \cap X|_\delta \approx \delta^{-1}$ for "typical" $Q \in \mathcal{D}_{\Delta}(X)$, we infer
$$|X|_{\delta} \gtrapprox \Delta^{-1} \cdot \delta^{-1} = \delta^{-3/2}.$$
Repeating this reasoning a few more times leads to $|X|_{\delta} \approx \delta^{-2}$, as desired. The rigorous version of this argument is the proof of Proposition \ref{prop1} based on Lemma \ref{lemma1}.

From  Proposition \ref{prop1} (and a "measure-theoretic" version of it stated in Proposition \ref{prop:main}) adapting a strategy developed in \cite{2023arXiv230903068O,Orponen2024Jan}, we derive Proposition \ref{prop:4.7}: convolution with a uniformly perfect measure $\sigma$ on $\mathbb{P}$ is $L^{2}$-flattening. In particular, taking repeated self-convolutions of $\sigma$ gradually and quantitatively decreases the $L^{2}$-norm of the convolution. This is formalised in Corollary \ref{cor2}. This easily yields Theorem \ref{thm:main}, see Section \ref{s:finalProof}.

\subsection{Acknowledgements} We are  grateful to Osama Khalil for  many discussions and insights. This project would not have materialised without Osama's input.

\section{Preliminaries}\label{s:preliminaries}

\subsection{Uniform sets} \label{Section: uniform sets} For $\delta \in 2^{-\mathbb{N}}$ and $P \subset \mathbb{R}^d$, let $\mathcal{D}_{\delta}(P)$ denote the collection of those cells from the $d$-dimensional dyadic partition  that intersect $P$. From now on we denote $|P|_\delta:=|\mathcal{D}_{\delta}(P)|$, which coincides with the $\delta$-covering number up to multiplicative constants. 
We  next recall (from e.g. \cite[Section 2]{2023arXiv230110199O}) the notion of \emph{uniform sets}:
\begin{definition}[Uniform set]\label{def:uniformity}
Let $n \geq 1$, and let
\begin{displaymath} \delta = \Delta_{n} < \Delta_{n - 1} < \ldots < \Delta_{1} \leq \Delta_{0} = 1 \end{displaymath}
be a sequence of dyadic scales.  A set $P\subset [0,1)^{d}$ is \emph{$\{\Delta_j\}_{j=1}^n$-uniform} if there is a sequence $\{N_j\}_{j=1}^n$ (called the \emph{branching numbers of $P$}) such that $N_{j} \in 2^{\N}$ and 
\begin{displaymath} |P\cap Q|_{\Delta_{j}} = N_j, \qquad j\in \{1,\ldots,n\}, \, Q\in\mathcal{D}_{\Delta_{j - 1}}(P). \end{displaymath}
We also extend this definition to $\mathcal{P}\subset\mathcal{D}_{\delta}([0,1)^{d})$ by applying it to $\cup\mathcal{P}$.
\end{definition}

The following Proposition  is \cite[Corollary 6.9]{2023arXiv230110199O}. It allows  one to (nearly) "exhaust" a set $\mathcal{P} \subset \mathcal{D}_{\delta}([0,1)^{d})$ by uniform sets. 

\begin{proposition}\label{cor1} For every $\epsilon > 0$, there exists $T_{0} = T_{0}(\epsilon) \geq 1$ such that the following holds for all $\delta = 2^{-mT}$ with $m \geq 1$ and $T \geq T_{0}$. Let $\mathcal{P} \subset \mathcal{D}_{\delta}([0,1)^{d})$. Then, there exist disjoint $\{2^{-jT}\}_{j = 1}^{m}$-uniform subsets $\mathcal{P}_{1},\ldots,\mathcal{P}_{N} \subset \mathcal{P}$ with the properties
\begin{itemize}
\item $|\mathcal{P}_{j}| \geq \delta^{2\epsilon}|\mathcal{P}|$ for all $1 \leq j \leq N$,
\item $|\mathcal{P} \, \setminus \, (\mathcal{P}_{1} \cup \ldots \cup \mathcal{P}_{N})| \leq \delta^{\epsilon}|\mathcal{P}|$.
\end{itemize}
\end{proposition}
\subsection{Rossi and Shmerkin's theorem} \label{Section:ross-shmekrin}
 A key tool in the proof of Theorem \ref{thm:main} will be Shmerkin's inverse theorem \cite[Theorem 2.1]{Sh}. In fact, we  use the theorem via the following Proposition \ref{RSProp} due to Rossi and Shmerkin. We require the following terminology:
\begin{definition}[$\delta$-measures and their $L^{2}$-norm]\label{def:delta measures and l2 norm} Let $\delta \in 2^{-\N}$. A probability measure $\nu$ on $\R$ is called a \emph{$\delta$-measure} if $\spt \nu \subset \delta \Z$. The $L^{2}$-norm of a $\delta$-measure is defined by
\begin{displaymath} \|\nu\|_{L^{2},\mathrm{Sh}}^{2} := \sum_{z \in \delta \Z} \nu(z)^{2}. \end{displaymath}
\end{definition}
Note that that the computation of $\|\nu\|_{L^{2},\mathrm{Sh}}^{2}$ depends implicitly on $\delta$.

Rossi and Shmerkin \cite[Proposition 3.1]{MR4163999} prove that convolution on $\R$ with a uniformly perfect measure (recall Definition \ref{def:uniformlyPerfect1}) results in a smaller $L^{2}$-norm, unless the starting position is already quite flat. We will require a slightly refined version of this result which allows for the following "relative" notion of uniform perfectness:
\begin{definition}[$(D,\beta,U)$-uniformly perfect measure]\label{def:uniformlyPerfect} Let $D > 1$, $\beta \in [0,1)$, and $U \subset \R^{d}$. A Radon measure $\sigma$ on $\R^{d}$ is called \emph{$(D,\beta,U)$-uniformly perfect} if $\diam(\spt \sigma) > 0$, and 
\begin{equation}\label{form25} \sigma(B(x,r)) \leq \beta \cdot \sigma(B(x,Dr))  \end{equation} 
for all balls $B(x,r)$ such that $\spt \sigma \not\subset B(x,Dr)$ and $B(x,Dr) \subset U$. We abbreviate $(D,\beta,\R^{d})$-uniform perfectness to $(D,\beta)$-uniform perfectness.

Slightly abusing terminology, a $\delta$-measure $\sigma$ on $\R$ is called $(D,\beta,U)$-uniformly perfect if \eqref{form25} holds for all $r \geq \delta$ (still assuming $\spt \sigma \not\subset B(x,Dr)$ and $B(x,Dr) \subset U$).
\end{definition}

\begin{remark} We need this refined definition since in our application of Rossi and Shmerkin's theorem, we shall require certain restrictions of the original measure to be uniformly perfect (see Lemma \ref{lemma3} below). Now, by \cite[Lemma 4.1]{MR4163999}, Ahlfors regular measures are always uniformly perfect. For such measures, it was shown by Bortz et. al. \cite[Lemma 2.1]{bortz2022corona} that it is possible to "localise" the measure without losing Ahlfors regularity. We do not know if a similar localisation is possible for uniformly perfect measures. \end{remark}

%To further illustrate how delicate the definition is, suppose $\sigma$ is a Borel probability measure that is $(D,\beta,[0,1])$-uniformly perfect measure with $\spt \sigma \subset [0,1]$. It is then not clear whether $\sigma$ is $(D',\beta', \mathbb{R})$-uniformly perfect for some $D' > 1$ and $\beta' \in [0,1)$. Indeed, $(D,\beta,[0,1])$-uniform perfectness gives no information about intervals centred at $\{0,1\}$. 

We proceed to state the refined version of  \cite[Proposition 3.1]{MR4163999}.

\begin{proposition}[Rossi-Shmerkin]\label{RSProp} For all $D > 1$ and $\beta,\eta,\mathfrak{d} \in (0,1)$, there exist $\epsilon = \epsilon(D,\beta,\eta) > 0$ and $\delta_{0} = \delta_{0}(D,\beta,\eta,\mathfrak{d}) > 0$ such that the following holds true.

Let $\delta \in 2^{-\N} \cap (0,\delta_{0}]$, and let $\mu,\sigma$ be $\delta$-measures such that:
\begin{itemize}
\item $\spt \mu \subset [0,1]$ and $\|\mu\|_{2,\mathrm{Sh}}^{2} \geq \delta^{1 - \eta}$,
\item $\sigma$ is $(D,\beta,[-2,2])$-uniformly perfect,
\item $\diam(\spt\sigma) \geq \mathfrak{d}$, and $\sigma([0,1]) \geq \delta^{\epsilon}$.
\end{itemize}
Then 
$$\|\mu \ast \sigma|_{[0,1]}\|_{2,\mathrm{Sh}} \leq \delta^{\epsilon}\|\mu\|_{2,\mathrm{Sh}}.$$ \end{proposition} 
We give the full details of the proof, even though they are virtually the same as the proof of \cite[Proposition 3.1]{MR4163999}.

\subsubsection{Proof of Proposition \ref{RSProp}} \label{appendix RSprop}
Let us first recall the original statement of Shmerkin's inverse theorem  \cite[Theorem 2.2]{MR4163999}. 

\begin{thm}[Shmerkin's inverse theorem]\label{thm:inverse} For each $\zeta \in (0,1]$, and $T_{0} \in \N$, there exist $T \geq T_{0}$ and $\epsilon = \epsilon(\zeta,T_{0}) > 0$ such that the following holds for $m \geq m_{0}(\zeta,T_{0}) \in \N$.

Let $\delta := 2^{-mT}$, and let $\mu,\sigma$ be $\delta$-measures with $\spt \mu,\spt \sigma \subset [0,1]$, and 
\begin{displaymath} \|\mu \ast \sigma\|_{L^{2},\mathrm{Sh}} \geq \delta^{\epsilon}\|\mu\|_{L^{2},\mathrm{Sh}}. \end{displaymath} Then, there exist sets $A \subset \spt \mu$, $B \subset \spt \sigma$, numbers $k_{A},k_{B} \in \delta \Z \cap [0,1)$, and a set $\mathcal{S} \subset \{1,\ldots,m\}$, such that
\begin{itemize}
\item[(A1)] $\|\mu|_{A}\|_{L^{2},\mathrm{Sh}} \geq \delta^{\zeta}\|\mu\|_{L^{2},\mathrm{Sh}}$.
\item[(A2)] $\mu(x) \leq 2\mu(y)$ for all $x,y \in A$.
\item[(A3)] $A' = A + k_{A} \subset [0,1)$ and $A'$ is $\{2^{-jT}\}_{j = 1}^{m}$-uniform with branching numbers $N_{j}^{A}$.
\item[(A4)] If $x \in A'$, $j \in \{0,\ldots,m - 1\}$, and $I \in \mathcal{D}_{2^{-jT}}$ is the interval containing $x$, then $x \in \tfrac{1}{2}I$. 
\item[(B1)] $\sigma(B) \geq \delta^{\zeta}$.
\item[(B2)] $\sigma(x) \leq 2\sigma(y)$ for all $y \in B$.
\item[(B3)] $B' = B + k_{B} \subset [0,1)$, and $B'$ is $\{2^{-jT}\}_{j = 1}^{m}$-uniform with branching numbers $N_{j}^{B}$. 
\item[(B4)] If $x \in B'$, $j \in \{0,\ldots,m - 1\}$, and $I \in \mathcal{D}_{2^{-jT}}$ is the interval containing $x$, then $x \in \tfrac{1}{2}I$.
\end{itemize}
Moreover:
\begin{itemize}
\item[(S1)] If $j \in \mathcal{S}$, then $N_{j}^{A} \geq 2^{(1 - \zeta)T}$, and if $j \notin \mathcal{S}$, then $N_{j}^{B} = 1$.
\item[(S2)] The set $\mathcal{S}$ satisfies
\begin{displaymath} \log \|\sigma\|_{L^{2},\mathrm{Sh}}^{-2} - \zeta \log \tfrac{1}{\delta} \leq T|\mathcal{S}| \leq \log \|\mu\|_{L^{2},\mathrm{Sh}}^{-2} + \zeta \log \tfrac{1}{\delta}. \end{displaymath}
\end{itemize} \end{thm}

Let us restate  Proposition \ref{RSProp}, this time with more convenient notation, before going into the proof:

\begin{proposition}\label{RSProp2} For all $\beta,\eta,\mathfrak{d},D > 0$, there exist $\epsilon = \epsilon(D,\beta,\eta) > 0$ and $\delta_{0} = \delta_{0}(D,\beta,\eta,\mathfrak{d}) > 0$ such that the following holds for all $\delta \in 2^{-\N} \cap (0,\delta_{0}]$. Let $\mu,\sigma$ be $\delta$-measures, where
\begin{itemize}
\item $\spt \mu \subset [0,1]$ and $\|\mu\|_{2,\mathrm{Sh}}^{2} \geq \delta^{1 - \eta}$,
\item  $\sigma$ is a $(2^{D},2^{-\beta},[-2,2])$-uniformly perfect,
\item $\diam(\spt\sigma) \geq \mathfrak{d}$, and $\sigma([0,1]) \geq \delta^{\epsilon}$.
\end{itemize}
Then 
$$\|\mu \ast \sigma|_{[0,1]}\|_{2,\mathrm{Sh}} \leq \delta^{\epsilon}\|\mu\|_{2,\mathrm{Sh}}.$$  \end{proposition} 

\begin{remark} The only difference to Proposition \ref{RSProp} is that we have renamed "$D,\beta$" in Proposition \ref{RSProp} to "$2^{D},2^{-\beta}$", therefore also replacing the hypothesis $D > 1$ by $D > 0$. \end{remark} 

\begin{proof}[Proof of Proposition \ref{RSProp2}] We start by fixing parameters. Let 
\begin{equation}\label{form47} \zeta := \min\{\eta \beta/(20D),\eta/2\} \end{equation}
and $T_{0} := \ceil{D + 4}$. Let $\epsilon_{0} > 0$, $T \geq T_{0}$, and $m_{0} \in \N$ be the parameters given by Theorem \ref{thm:inverse} with these constants, and let $\epsilon := \max\{\epsilon_{0},\zeta/2\}$. Let $m_{1} \geq \N$ be the smallest integer such that $2^{-m_{1}} < \mathfrak{d}$. Assume that $m \geq m_{0}$, and also
\begin{equation}\label{form46} m \geq 4m_{1}/\eta. \end{equation}
We first consider the special case where the scale $\delta \in 2^{-\N}$ has the form $\delta = 2^{-mT}$ for $m \geq m_{0}$, as above. We will  relax this hypothesis at the end of the proof. So, fix $m \geq m_{0}$, write $\delta := 2^{-mT}$, and let $\mu,\sigma$ be $\delta$-measures as in the statement of Proposition \ref{RSProp2}.

Assume, towards a contradiction, that
\begin{displaymath} \|\mu \ast \sigma|_{[0,1]}\|_{L^{2},\mathrm{Sh}} \geq  \delta^{\epsilon}\|\mu\|_{L^{2},\mathrm{Sh}}. \end{displaymath}
Since $\sigma$ is a probability measure, 
$$\|\mu \ast \sigma([0,1])^{-1}\sigma|_{[0,1]}\|_{L^{2},\mathrm{Sh}} \geq \|\mu \ast \sigma|_{[0,1]}\|_{L^{2},\mathrm{Sh}} \geq  \delta^{\epsilon}\|\mu\|_{L^{2},\mathrm{Sh}} \geq  \delta^{\epsilon_{0}}\|\mu\|_{L^{2},\mathrm{Sh}}.$$ 
The hypotheses of Theorem \ref{thm:inverse} are thus  met by the probability measures $\mu$ and $\sigma([0,1])^{-1}\sigma|_{[0,1]}$. We therefore obtain sets $A \subset \spt \mu$ and $B \subset \spt \sigma \cap [0,1]$,  translations $k_{A},k_{B} \in \delta \Z \cap [0,1)$, and a set $\mathcal{S} \subset \{1,\ldots,m\}$ corresponding to the "full branching" scales of $A$. We write $\sigma_{B}$ for the translate of $\sigma$ by $k_{B}$, thus
\begin{displaymath} \sigma_{B}(E) := \sigma(E - k_{B}), \qquad E \subset \R. \end{displaymath}

The following Claim is where the $(2^D, 2^{-\beta}, [-2,2])$-uniform perfectness of $\sigma$ is applied:
\begin{claim} \label{claim I1-I3} Let $j \in \{0,\ldots,m - 1\}$, and let $I,J \subset \R$ be intervals satisfying the following:
\begin{itemize}
\item[(I1)] $I \in \mathcal{D}_{2^{-(j + 1)T}}([0,1))$ and $J \in \mathcal{D}_{2^{-jT}}([0,1))$.
\item[(I2)] $I \cap \tfrac{1}{2}J \neq \emptyset$.
\item[(I3)] $\spt \sigma_B \not\subset J$.
\end{itemize}
Then,
\begin{equation}\label{form48} \sigma_{B}(I) \leq 2^{-\beta(T - 2)D}\sigma_{B}(J). \end{equation} \end{claim}

\begin{proof} Recall that $T \geq T_{0} \geq D + 5$. This means that $I$ is a lot shorter than $J$. In particular, the hypothesis $I \cap \tfrac{1}{2}J \neq \emptyset$ implies that even the concentric $2^{D}$-thickening $2^{D}I$ is contained in $J$. More generally, for $n \in \N$, we have $2^{Dn}I \subset J$ as long as $\ell(2^{Dn}I) \leq \tfrac{1}{4}\ell(J)$, or equivalently
\begin{displaymath} 2^{Dn - (j + 1)T} \leq  2^{-jT - 2} \quad \Longleftrightarrow \quad Dn \leq T - 2 \quad \Longleftrightarrow \quad n \leq (T - 2)/D. \end{displaymath} 
In particular, whenever $2^{Dn}I \subset J$, we have $2^{Dn}I \subset [0,1)$. Therefore,
$$2^{Dn}I - k_{B} \subset [-2,2].$$ 
So, since 
$$2^{Dn}(I - k_{B}) = 2^{Dn}I - k_{B}$$
we can apply $n$-times the $(2^{D},2^{-\beta},[-2,2])$-uniform perfectness of $\sigma$  to obtain
\begin{displaymath} \sigma_{B}(I) = \sigma(I - k_{B}) \leq 2^{-\beta n}\sigma(2^{Dn}I - k_{B}) = 2^{-\beta n}\sigma_{B}(2^{Dn}I) \leq 2^{-\beta n}\sigma_{B}(J). \end{displaymath}  
Since the inclusion $2^{Dn} I \subset J$ holds for at least when $n \leq (T - 2)/D$, we deduce \eqref{form48}. \end{proof}

Recall that $m_{1} \in \N$ is the smallest integer such that $2^{-m_{1}} < \mathfrak{d} \leq \diam(\spt \sigma)$. Let 
\begin{displaymath} \mathcal{N} := \{m_{1} \leq j \leq m : N_{j}^{B} = 1\}. \end{displaymath}
Recall that we assume $\|\mu\|_{L^{2},\mathrm{Sh}}^{2} \geq \delta^{1 - \eta}$. So, by Theorem \ref{thm:inverse} (S2) 
\begin{equation}\label{form45} T|\mathcal{S}| \leq \log \|\mu\|_{L^{2},\mathrm{Sh}}^{-2} + \zeta \log \tfrac{1}{\delta} \leq (1 - \eta + \zeta)\log \tfrac{1}{\delta} \stackrel{\eqref{form47}}{\leq} (1 - \eta/2)mT.  \end{equation}
Therefore, by Theorem \ref{thm:inverse} (S1) 
\begin{equation}\label{form49} |\mathcal{N}| \geq m - |\mathcal{S}| - m_{1} \geq (\eta/2)m - m_{1} \stackrel{\eqref{form46}}{\geq} (\eta/4)m \stackrel{\eqref{form47}}{>} 4\zeta Dm/\beta. \end{equation}

Let $j \in \mathcal{N}$ and $I \in \mathcal{D}_{2^{-(j + 1)T}}(B + k_{B})$, and let $J \in \mathcal{D}_{2^{-jT}}(B + k_{B})$ be the parent of $I$.  We now argue that $I,J$ satisfy conditions (I1)-(I3) of Claim \ref{claim I1-I3}. Indeed, $B + k_{B} \subset [0,1)$ by Theorem \ref{thm:inverse} (B3), so $I,J \subset [0,1)$. Second, Theorem \ref{thm:inverse} (B4) implies that $I \cap \tfrac{1}{2}J \neq \emptyset$. Finally, since $j \geq m_{1}$, we have 
$$2^{-jT} \leq 2^{-m_{1}T} < \mathfrak{d} = \diam(\spt \sigma)=\diam(\spt \sigma_B),$$
so $\spt \sigma_B \not\subset J$.

Thus, by \eqref{form48} we have that $\sigma_{B}(I) \leq 2^{-\beta(T - 2)/D}\sigma_{B}(J)$. 
Applying this estimate for all $j \in \mathcal{N}$ and  using that $N_{j}^{B} = 1$, we obtain
\begin{displaymath} \sigma_{B}(B + k_{B}) \leq 2^{-\beta(T - 2)|\mathcal{N}|/D} \stackrel{\eqref{form49}}{\leq} 2^{-4\zeta(T - 2)m} \leq 2^{-2\zeta m} = \delta^{2\zeta}.  \end{displaymath} 

On the other hand, by Theorem \ref{thm:inverse} (B1), and the assumptions $\epsilon \geq \zeta/2$ and $\sigma([0,1]) \geq \delta^{\epsilon}$,  we have 
\begin{displaymath} \sigma_{B}(B + k_{B}) = \sigma(B) \geq \sigma([0,1]) \cdot (\sigma([0,1])^{-1}\sigma|_{[0,1]})(B) \geq \delta^{\zeta + \epsilon} \geq \delta^{3\zeta/2}. \end{displaymath}
The last two displayed equations are contradictory. The proof is thus complete in the case where the scale $\delta$ has the special form $\delta = 2^{-mT}$. 

For the remaining cases, suppose now $\delta=2^{-mT-j}$ where $j\in \lbrace 1,...,T-2 \rbrace$. We retain the assumptions $m\geq m_0$ and \eqref{form46}. Let $\rho$ be a $\delta$-measure, and let $\rho^{(mT)}$ denote the corresponding level-$mT$ discretization of $\rho$. That is
$$\rho^{(mT)}(z) = \rho \left( [z, z+2^{-mT}) \right),\quad z\in 2^{-mT} \mathbb{Z}.$$
Then
$$\|\rho\|_{2,\mathrm{Sh}} = \sum_{I \in D_{2^{-mT}}} \rho(I)^2 \sum_{J \subseteq D_{\delta},\, J\subseteq I} \left( \frac{\rho(J)}{\rho(I)} \right)^2.  $$
Therefore, 
$$\|\rho\|_{2,\mathrm{Sh}} \leq  \| \rho^{(mT)}  \|_{2,\mathrm{Sh}} \leq 2^{ \frac{T-1}{2}}  \|\rho\|_{2,\mathrm{Sh}}.  $$ 

Let us quickly verify that $\sigma^{(mT)}$ satisfies the conditions of the Proposition (with slightly adjusted parameters), assuming $\sigma$ does (we retain the same $\mu$ throughout):
\begin{itemize}

\item First, we may increase $m_0$ so that $2^{-mT} < \mathfrak{d}$ for all $m>m_0$. This ensures the non-triviality of the discretization.  Next,  since the $\delta$-measure $\sigma$  is assumed to be  $(2^{D},2^{-\beta},[-2,2])$-uniformly perfect, for every $r\geq 2^{-mT} > \delta$ we have
$$ \sigma^{(mT)} \left( B(x,r) \right) \leq \sigma \left( B(x,2r) \right) \leq 2^{-\beta} \sigma \left( B(x,2\cdot 2^D \cdot r) \right)$$
$$ \leq 2^{-\beta} \sigma^{(mT)} \left( B(x, \left( 2\cdot 2^D+1 \right) \cdot r) \right)$$
for all balls $B(x,r)$ such that $\spt \sigma \not\subset B(x,(2^{D+1}+1)r)$ and $$B(x,\left( 2\cdot 2^D+1 \right) \cdot r) \subset [-2,2].$$ 
This shows that the $2^{-mT}$-measure $\sigma^{(mT)}$  is   $(2^{D+1}+1,2^{-\beta},[-2,2])$-uniformly perfect.

\item Since $\diam(\spt\sigma) \geq \mathfrak{d}$ the same is true for $\sigma^{(mT)}$ (perhaps up to a uniform multiplicative constant). In addition, $\sigma^{(mT)} ([0,1]) = \sigma([0,1]) \geq \delta^{\epsilon} \geq C_T (2^{-mT})^\epsilon$, where $C_T$ is some global multiplicative constant (that depends only $T$ and therefore only on $D$).

\end{itemize}
Thus,   \cite[Lemma 2.1]{MR4163999} provides us with constant $C_2$ as below (that depends only on the ambient dimension, which is one in our case),  and applying the already established case $\delta=2^{-mT}$ (with the parameters as in the bullets above) we conclude that:
$$\|\mu \ast \sigma|_{[0,1]}\|_{L^{2},\mathrm{Sh}} \leq C_2 \| \left( \mu \ast \sigma|_{[0,1]} \right)^{(mT)} \|_{L^{2},\mathrm{Sh}} \leq C_2 \|\mu^{(mT)} \ast (\sigma|_{[0,1]} ) ^{(mT)} \|_{L^{2},\mathrm{Sh}} $$
$$\leq  C_2 \|\mu^{(mT)} \ast ((\sigma ^{(mT)})|_{[0,1]} )  \|_{L^{2},\mathrm{Sh}} \leq C_2\cdot  \delta^{\epsilon}\| \mu^{(mT)} \|_{2,\mathrm{Sh}} $$
$$\leq C_2\cdot \delta^{\epsilon}\cdot  2^{ \frac{T-1}{2}} \|\mu \|_{L^{2},\mathrm{Sh}} \leq  \delta^{\epsilon/2} \|\mu \|_{L^{2},\mathrm{Sh}}.$$
 The latest inequality is true if $m$ is taken sufficiently large in manner dependent only on the fixed parameters as above, $C_2$, and $2^{ \frac{T-1}{2}}$. Note also the use of pointwise inequality $(\sigma|_{[0,1]} ) ^{(mT)} \leq (\sigma ^{(mT)})|_{[0,1]}$.  The proof is complete.
\end{proof}

\subsection{Geometry of uniformly perfect measures} \label{Section:uniformly perfect measures} In this Section we study uniformly perfect measures, as in Definition \ref{def:uniformlyPerfect}. First, we note that uniform perfectness is invariant under push-forwards of similarity maps. Recall that $T:\mathbb{R}^d \rightarrow \mathbb{R}^d$ is called a similarity map if $|T(x) - T(y)| = \lambda |x - y|$ for some $\lambda \in (0,\infty)$.
\begin{lemma}\label{lemma4} let $D > 1$ and $\beta \in [0,1)$. Let $T \colon \R^{d} \to \R^{d}$ be a similarity map. If $\sigma$ is $(D,\beta,U)$-uniformly perfect, then $T\sigma$ is uniformly $(D,\beta,V)$-uniformly perfect with $V = T(U)$.
\end{lemma}

\begin{proof} Fix $y \in \R^{d}$ and $r > 0$ such that 
$$\spt T \sigma \not\subset B(y,Dr) \text{ and } B(y,Dr) \subset V.$$ 
Write $x := T^{-1}(y)$, thus $T^{-1}(B(y,Dr)) = B(x,Dr/\lambda)$. Then 
\begin{displaymath} \spt \sigma \not\subset B(x,Dr/\lambda) \quad \text{and} \quad B(x,Dr/\lambda) \subset U, \end{displaymath}
so the $(D,\beta,U)$-uniform perfectness of $\sigma$ yields
\begin{displaymath} (T\sigma)(B(y,r)) = \sigma(B(x,r/\lambda)) \leq \beta \cdot B(x,Dr/\lambda) = \beta \cdot  (T\sigma)(B(y,Dr)). \end{displaymath}
Therefore $T\sigma$ is $(D,\beta,V)$-uniformly perfect. \end{proof} 

\begin{remark} To prove that a measure $\sigma$ is uniformly perfect, it suffices to consider balls centred at points $x\in \spt \sigma$. This simple fact is proved in \cite[Section 3.1]{MR4163999}. \end{remark} 

Next, we show that uniformly perfect measures are always Frostman measures. 
\begin{lemma}\label{frostmanLemma} Let $\sigma$ be a Borel probability measure on $\R^{n}$. If $\sigma$ is  $(D,\beta)$-uniformly perfect, then $\sigma$ is a Frostman measure:
\begin{equation}\label{form39} \sigma(B(x,r)) \leq (2D)^{s}\diam(\spt \sigma)^{-s} \cdot r^{s}, \qquad x \in \R^{d}, \, r > 0, \end{equation}
where $s = -\log \beta/\log D > 0$.
\end{lemma}

\begin{proof}  Assume $\sigma$ is  $(D,\beta)$-uniformly perfect.  First, we prove \eqref{form39} under the additional assumptions that $\diam(\spt \sigma) > 2$ and $r\in (0,1]$.  Write $D = 2^{d}$ and $\beta = 2^{-b}$ (so $d = \log D$ and $b = -\log \beta$). Fix $x \in \R^{d}$ and $n \in \N \cup \{0\}$. Since $\spt \sigma \not\subset B(x,1) = B(x,D^{0})$, by applying uniformly perfectness $n$-times we have
\begin{displaymath} \sigma(B(x,2^{-dn})) = \sigma(B(x,D^{-n})) \leq \beta^{n} \cdot \sigma(B(x,D^{0})) \leq  2^{-b n}. \end{displaymath}
That is, recalling that $s = -\log \beta/\log D > 0$,
$$\sigma(B(x,2^{-dn})) \leq  (2^{-dn})^{s}.$$ 
For general $r \in (0,1]$ choose $n \in \N \cup \{0\}$ such that $D^{-n - 1} \leq r \leq D^{-n}$. Then the previously displayed equation yields
\begin{equation}\label{form38} \sigma(B(x,r)) \leq \sigma(B(x,D^{-n})) \leq (D^{-n})^{s} \leq D^{s}r^{s}, \qquad r \in (0,1]. \end{equation}

Next, assume  $\diam(\spt \, \sigma) > 0$ is arbitrary, but  consider only $0 < r < \tfrac{1}{2}\diam(\spt \sigma)$. Fix $\mathfrak{d} < \diam(\spt \sigma)$ such that $r \leq \mathfrak{d}/2$. let $T_{\mathfrak{d}}$ be the dilation $T_{\mathfrak{d}}(x) := 2x/\mathfrak{d}$. By Lemma \ref{lemma4}, $T_{\mathfrak{d}}\sigma$ is a $(D,\beta)$-uniformly perfect measure with $\diam(\spt T_{\mathfrak{d}}\sigma) > 2$. Since $2r/\mathfrak{d} \leq 1$, we can apply \eqref{form38} and see that
\begin{displaymath} \sigma(B(x,r)) = (T_{\mathfrak{d}}\sigma)[B(T_{\mathfrak{d}}(x),2r/\mathfrak{d})] \leq D^{s}(2/\mathfrak{d})^{s}r^{s}. \end{displaymath}
Letting $\mathfrak{d} \nearrow \diam(\spt \sigma)$ completes the proof in the case $0 < r < \tfrac{1}{2}\diam(\spt \sigma)$.

Finally, if $r \geq \tfrac{1}{2}\diam(\spt \sigma)$, the inequality \eqref{form39} follows from the trivial estimate $\sigma(B(x,r)) \leq 1 \leq (2r/\diam(\spt \sigma))^{s}$, and the hypothesis $D > 1$. \end{proof}

We next show that restrictions of uniformly perfect measures are uniformly perfect:

\begin{lemma}\label{lemma3} Let $D > 1$, $\beta \in [0,1)$, and $U\subseteq \mathbb{R}^d$, and let $\sigma$ be $(D,\beta,U)$-uniformly perfect. Let $V \subset U$ be a Borel subset. Then $\sigma|_{V}$ is $(D,\beta,V)$-uniformly perfect. \end{lemma}

\begin{proof} Let $x \in \R^{d}$ and $r > 0$ be such that 
$$\spt \sigma|_{V} \not\subset B(x,Dr) \text{ and } B(x,Dr) \subset V.$$
Then evidently $\spt \sigma \not\subset B(x,Dr)$, and $B(x,Dr) \subset U$. Applying the uniform perfectness of $\sigma$ on $U$ yields 
$$\sigma(B(x,r)) \leq \beta \cdot \sigma(B(x,Dr)).$$
Since $B(x,r) \subset B(x,Dr) \subset V$, the same inequality remains true for $\sigma|_{V}$.  \end{proof}

\subsection{Further auxiliary lemmas}
In this Section we discuss some further standard geometric results needed in the proof of Theorem \ref{thm:main}. Recall the definition of $\delta$-measures and their $L^2$ norms from Definition \ref{def:delta measures and l2 norm}. 
\begin{claim} \label{rem1}
Let $\nu$ be a $\delta$-measure. If $\nu$ has a constant density then $\nu$ is a uniform measure on $\spt \nu$, and
$$\|\nu\|_{L^{2},\mathrm{Sh}}^{2}= | \spt \nu |^{-1}.$$
\end{claim}
This is a simple consequence of the assumptions, and that $\delta$-measures are always assumed to be probability measures. We omit the details.

The next lemma gives a sufficient criterion to check that the convolution of two $\delta$-measures has large $L^{2}$-norm:
\begin{lemma}\label{lemma2} Let $c > 0$ and $C \geq 1$, and let $\mu,\sigma$ be $\delta$-measures on $\R$. Let $G \subset \delta \Z \times \delta \Z$ be a set. Assume that:
\begin{enumerate}
\item $\mu$ has constant density on $X := \spt \mu \subset \delta \Z \cap [0,1]$; and,

\item The set $G$ satisfies
\begin{displaymath} (\mu \times \sigma)(G) \geq c \quad \text{and} \quad |\{x + y : (x,y) \in G\}| \leq C|X|. \end{displaymath} 
\end{enumerate}
Then, 
$$\|\mu \ast \sigma\|_{L^{2},\mathrm{Sh}} \geq (c/\sqrt{C})\|\mu\|_{L^{2},\mathrm{Sh}}.$$
\end{lemma} 

\begin{proof} Let us write $Z := \{x + y : (x,y) \in G\}$ and estimate as follows:
\begin{align*} \|\mu \ast \sigma\|_{L^{2},\mathrm{Sh}}^{2} & \stackrel{\mathrm{def.}}{=} \sum_{z \in \delta \Z} (\mu \ast \sigma)(z)^{2}\\
& = \sum_{z \in \delta \Z} (\mu \times \sigma)(\{(x,y) \in \delta \Z \times \delta \Z : x + y = z\})^{2}\\
& \geq \sum_{z \in Z} (\mu \times \sigma)(\{(x,y) \in G : x + y  = z\})^{2}\\
& \geq \tfrac{1}{|Z|} \Big( \sum_{z \in Z} (\mu \times \sigma)(\{(x,y) \in G : x + y = z\}) \Big)^{2} \\
&\geq  \frac{c^{2}}{C|X|} = \frac{c^{2}}{C} \cdot  \|\mu\|_{L^{2},\mathrm{Sh}}^{2}. \end{align*} 
Note the use of Cauchy-Schwarz in the fourth inequality, and the use of Claim \ref{rem1} for the last equality.
\end{proof} 
Finally, we require the following Lemma about push-forwards and product measures:
\begin{lemma}\label{lemma5} Let $\nu,\sigma$ be finite Radon measures on $\R^{d}$, and let $f,g \colon \R^{d} \to \R^{d}$ be Borel. Then $f\mu \times g \sigma = (f \times g)(\mu \times \sigma)$. In particular 
\begin{equation}
\label{eq:in particular} (f\mu \times g\sigma)[(f \times g)(B)] \geq (\mu \times \sigma)(B), \qquad B \in \mathrm{Bor}(\R^{2d}).
\end{equation}
\end{lemma} 

\begin{proof} The first claim implies \eqref{eq:in particular} by noting that
\begin{displaymath} (f\mu \times g\sigma)(f \times g) = (\mu \times \sigma)(f \times g)^{-1}[(f \times g)(B)] \geq (\mu \times \sigma)(B). \end{displaymath}
To prove the first claim, note that $\mathrm{Bor}(\R^{2d}) = \mathrm{Bor}(\R^{d}) \times \mathrm{Bor}(\R^{d})$ is the $\sigma$-algebra generated by the $\pi$-system of rectangles $A \times B$, $A,B \in \mathrm{Bor}(\R^{d})$. The two measures $f\mu \times g\sigma$ and $(f \times g)(\mu \times \sigma)$ clearly agree on this $\pi$-system, and have common (finite) mass. So, by either Dynkin's lemma or the monontone class Lemma \cite[Lemma 2.35]{Folland1999},  they agree on $\mathrm{Bor}(\R^{2d})$. \end{proof} 

\section{Main technical proposition} \label{Section: main tech}
The purpose of this section is to state and prove our technical result, Proposition \ref{prop:main}, for which we gave some exposition in Section \ref{Section:sketch}.
Adapting some arguments  from \cite{Orponen2024Jan}, it will form the key step towards the proof of Theorem \ref{thm:main}. 

We start by introducing further notation. Let $\psi = \psi_{d} \in C^{\infty}_{c}(\R^{d})$ be a fixed radially decreasing function satisfying $\int \psi = 1$ and $\mathbf{1}_{B(1/2)} \leq \psi \lesssim \mathbf{1}_{B(1)}$. Starting from $\psi$, define the standard "approximate identity" family $\{\psi_{\delta}\}_{\delta > 0}$, where $\psi_{\delta}(x) = \delta^{-d}\psi(x/\delta)$. For a Radon measure $\mu$ on $\R^{d}$, we write $\mu_{\delta} := \mu \ast \psi_{\delta}$.

 Here is our main technical proposition:
 
\begin{proposition} \label{prop:main} For all $\alpha \in (0,2)$, $\beta \in [0,1)$, $\mathfrak{d} > 0$, and $D > 1$ there exist $\epsilon = \epsilon(\alpha,\beta,D) > 0$ and $\delta_{0} = \delta_{0}(\alpha,\beta,\mathfrak{d},D) > 0$ such that the following holds for all $\delta \in (0,\delta_{0}]$. 

Let $\mu, \sigma$ be Radon measures, and let $E \subset \R^{2}$ be Borel set such that:
\begin{enumerate}
\item $\spt \mu$ is contained in a dyadic cube of side length $1$, $\mu(\R^{2}) \leq 1$, and $\|\mu_{\delta}\|_{2}^{2} \leq \delta^{\alpha - 2 - \epsilon}$;

\item  $\sigma$ is $(D,\beta)$-uniformly perfect, $\sigma(\R^{2}) \leq 1$, $\spt \sigma \subset \mathbb{P}$ and $\diam(\spt \sigma) \geq \mathfrak{d}$;

\item $(\mu \ast \sigma)(E) \geq \delta^{\epsilon}$.
\end{enumerate}
Then, 
$$|E|_{\delta} \geq \delta^{-\alpha - \epsilon}.$$
\end{proposition} 

Note that the exponent $\epsilon$ is independent of the diameter constant $\mathfrak{d}$. The notation $\|\cdot\|_{2}$ refers to the usual $L^{2}$-norm (as opposed to the notation $\|\cdot\|_{L^{2},\mathrm{Sh}}$ from Definition \ref{def:delta measures and l2 norm}).

We proceed with the proof of Proposition \ref{prop:main}. Our first goal is to reduce Proposition \ref{prop:main} to the following lemma about the growth of sumsets:

\begin{lemma}\label{lemma1} For every $\alpha \in [0,2)$, $\beta \in [0,1)$, $\mathfrak{d} > 0$, $D > 1$, and $T \in \N$, there exist $\epsilon = \epsilon(\alpha,\beta,D) > 0$ and $m_{0} = m_{0}(\alpha,\beta,\mathfrak{d},D,T) \in \N$ such that the following holds for all $\delta > 0$ of the form $\delta = 2^{-mT}$, where $m \geq m_{0}$.

Suppose we are given:
\begin{enumerate}
\item A $\{2^{-jT}\}_{j = 1}^{m}$- uniform set $\mathcal{X} \subset \mathcal{D}_{\delta}(K)$, where $K$ is a dyadic cube of side length $1$,  satisfying 
\begin{equation}\label{eq:sizeXloc} |\mathcal{X}  \cap Q| \leq \delta^{-\alpha/2}, \qquad Q \in \mathcal{D}_{\sqrt{\delta}}(\mathcal{X} ). \end{equation}

\item  A $(D,\beta)$-uniformly perfect probability measure $\sigma$ with $\spt \sigma \subset \mathbb{P}$ and $\diam(\spt \sigma) \geq \mathfrak{d}$.

\item A set $\mathcal{G} \subset \mathcal{X}  \times \mathcal{D}_{\delta}(\spt \sigma)$ such that, for  the uniform probability measure $\nu$ on $\cup \mathcal{X}$,
$$(\nu \times \sigma)(\cup \mathcal{G}) \geq \delta^{\epsilon}.$$

\end{enumerate}
 Then 
\begin{displaymath} |\{x + y : (x,y) \in \cup \mathcal{G}\}|_{\delta} \geq \delta^{-\epsilon}|\mathcal{X}| \end{displaymath}
\end{lemma}

First, we deduce Proposition \ref{prop:main} from Proposition \ref{prop1} below. Lemma \ref{lemma1} is otherwise the same statement as Proposition \ref{prop1}, except that the "global" size hypothesis \eqref{eq:sizeX} is replaced by a "local" counterpart \eqref{eq:sizeXloc}.

\begin{proposition}\label{prop1} For every $\alpha \in [0,2)$, $\beta \in [0,1)$, $\mathfrak{d} > 0$, $D > 1$, and $T \in \N$, there exist $\epsilon = \epsilon(\alpha,\beta,D) > 0$ and $m_{0} = m_{0}(\alpha,\beta,\mathfrak{d},D,T) \in \N$ such that the following holds for all $\delta > 0$ of the form $\delta = 2^{-mT}$, where $m \geq m_{0}$.

Suppose we are given:
\begin{enumerate}
\item A $\{2^{-jT}\}_{j = 1}^{m}$- uniform set $\mathcal{X} \subset \mathcal{D}_{\delta}(Q)$, where $Q$ is a dyadic cube of side length $1$,  satisfying 
\begin{equation}\label{eq:sizeX} |\mathcal{X}| \leq \delta^{-\alpha}. \end{equation}

\item  A $(D,\beta)$-uniformly perfect probability measure $\sigma$ supported on $\mathbb{P}$, with $\diam(\spt \sigma) \geq \mathfrak{d}$.

\item A set $\mathcal{G} \subset \mathcal{X}  \times \mathcal{D}_{\delta}(\spt \sigma)$ such that, for  the uniform probability measure $\nu$ on $\cup \mathcal{X}$,
$$(\nu \times \sigma)(\cup \mathcal{G}) \geq \delta^{\epsilon}.$$

\end{enumerate}
Then,
\begin{equation}\label{form5} |\{x + y : (x,y) \in \cup \mathcal{G}\}|_{\delta} \geq \delta^{-\epsilon}|\mathcal{X}|. \end{equation}
\end{proposition}
We proceed to show that Proposition \ref{prop1} formally implies Proposition \ref{prop:main}

\begin{proof}[Proof of Proposition \ref{prop:main} assuming Proposition \ref{prop1}] Write 
\begin{equation} \label{eq:parameters}
\gamma := \tfrac{1}{2}(\alpha + 2) \in (\alpha,2),
\end{equation}
and let $\epsilon_{0} = \epsilon_{0}(\gamma,\beta,D) > 0$ be the parameter given by Proposition \ref{prop1}. Assume that $\mu, Q, \sigma, E$ are as in Proposition \ref{prop:main} (1)-(3), satisfying those hypotheses with  respect to
\begin{equation}\label{form21} \epsilon := \epsilon(\alpha,\beta) := \tfrac{1}{20}\min\{\epsilon_{0},2 - \alpha\}. \end{equation}
We claim that $|E|_{\delta} \geq \delta^{-\alpha - \epsilon}$.

To apply Proposition \ref{prop1}, we need to extract a useful uniform set $\mathcal{X} \subset \mathcal{D}_{\delta}(\spt \mu)$. By hypothesis
\begin{displaymath} (\mu \times \sigma)(\{(x,y) : x + y \in E\}) = (\mu \ast \sigma)(E) \geq \delta^{\epsilon}. \end{displaymath} 
Let  
$$G := \{(x,y) : x + y \in E\}.$$
Given $\rho \in 2^{-\N}$, let $\mu^{\rho}$ be the restriction of $\mu$ to those squares $p \in \mathcal{D}_{\delta}(Q)$ such that 
$$\mu(p) \in [\rho,2\rho).$$ 
We claim that there exists some $\rho \in 2^{-\N}$ such that 
\begin{equation} \label{eq: dyadic pigeon}
(\mu^{\rho} \times \sigma)(G) \geq \delta^{2\epsilon}.
\end{equation}

Indeed, this follow since $\mu = \sum_{\rho} \mu^{\rho}$,  by dyadic pigeonholing, and by assuming $\delta > 0$ is sufficiently small in terms of $\epsilon$.

Our next step is to apply Proposition  \ref{cor1} with parameter $3\epsilon$ to the set 
$$\mathcal{P} := \mathcal{D}_{\delta}(\spt \mu^{\rho}).$$ 
We thus obtain a parameter $T_{0} = T_{0}(3\epsilon) \geq 1$, and a sequence of disjoint  $\{2^{-jT}\}_{j = 1}^{m}$-uniform sets $\mathcal{P}_{1},\ldots,\mathcal{P}_{N} \subset \mathcal{P}$ with the properties
\begin{itemize}
\item[(a)] $|\mathcal{P}_{j}| \geq \delta^{6\epsilon}|\mathcal{P}|$ for all $1 \leq j \leq N$ (in particular $N \leq \delta^{-6\epsilon}$); and,
\item[(b)] $|\mathcal{P} \, \setminus \, (\mathcal{P}_{1} \cup \ldots \cup \mathcal{P}_{N})| \leq \delta^{3\epsilon}|\mathcal{P}|$.
\end{itemize}
Note also that since $\mu_{\rho}(p) \sim \rho$ for all $p \in \mathcal{P}$ and $\mu$ has total mass less than $1$, we have
\begin{equation} \label{eq: rho leq}
\rho \cdot |\mathcal{P}|\leq \mu \left( \cup \mathcal{P} \right)\leq 1.
\end{equation}

Writing $\mathcal{R} := \mathcal{P} \, \setminus \, (\mathcal{P}_{1} \cup \ldots \cup \mathcal{P}_{N})$, by \eqref{eq: rho leq} and Part (b) above we have
$$\mu^{\rho}(\cup \mathcal{R}) \lesssim \delta^{3\epsilon}.$$
Consequently 
$$(\mu^{\rho} \times \sigma)((\cup \mathcal{R} \times \R^{2}) \cap G) \lesssim \delta^{3\epsilon},$$ 
and in particular, assuming $\delta > 0$ is sufficiently small in terms of $\epsilon$.
\begin{displaymath} (\mu^{\rho} \times \sigma)((\cup \mathcal{R} \times \R^{2}) \cap G) \leq \tfrac{1}{2}(\mu^{\rho} \times \sigma)(G). \end{displaymath}

It therefore follows from  (a) above and \eqref{eq: dyadic pigeon} that there exists $j \in \{1,\ldots,N\} \subset \{1,\ldots,\delta^{-6\epsilon}\}$ such that, writing $\mathcal{X} := \mathcal{P}_{j}$ and $X := \cup \mathcal{X}$,
\begin{equation}\label{form20} (\mu^{\rho} \times \sigma)((X \times \R^{2}) \cap G) \gtrsim \delta^{6\epsilon}(\mu^{\rho} \times \sigma)(G) \geq \delta^{8\epsilon}. \end{equation} 
Since $\|\mu_{\delta}\|_{2}^{2} \leq \delta^{\alpha - 2-\epsilon}$ by hypothesis, and $\mu_{\delta}(X) \geq \mu^{\rho}(X) \gtrsim \delta^{8\epsilon}$ (by \eqref{form20}), we may infer from Cauchy-Schwarz that
\begin{displaymath} \delta^{8\epsilon} \lesssim \int_{X} \mu_{\delta} \leq (\mathrm{Leb}(X))^{1/2}\delta^{(\alpha - 2 - \epsilon)/2} \quad \Longrightarrow \quad \mathrm{Leb}(X) \gtrsim \delta^{2 - \alpha + 17\epsilon}. \end{displaymath}
Consequently, provided $\delta > 0$ is small enough in terms of $\epsilon$,
\begin{equation}\label{form19} |\mathcal{X}| \geq \delta^{-\alpha + 18\epsilon}. \end{equation}
Now, write 
$$\mathcal{G} := \{(p,\theta) \in \mathcal{X} \times \mathcal{D}_{\delta}(\spt \sigma) : (p \times \theta) \cap G \neq \emptyset\}.$$ 
Let $\nu$ be the uniform probability on $X$ (as in Proposition \ref{prop1} Part (3)).
It follows from \eqref{form20}, and from the density constancy $\mu_{\rho}(p)/\nu(p) \sim \rho/|\mathcal{X}|$ for $p \in \mathcal{X}$, that 
\begin{displaymath} (\nu \times \sigma)(\cup \mathcal{G}) \gtrsim \delta^{8\epsilon}. \end{displaymath}

Since $8\epsilon < \epsilon_{0}$, for $\delta > 0$ small enough, 
$$(\nu \times \sigma)(\cup \mathcal{G}) > \delta^{\epsilon_{0}}.$$ 
We also recall that $\mathcal{X}$ is $\{2^{-jT}\}_{j = 1}^{m}$-uniform.  So, to apply Proposition \ref{prop1} we require an upper bound on $|\mathcal{X}|$ as  in \eqref{eq:sizeX}. Fortunately, we can deal with the case of "large $\mathcal{X}$" by an elementary argument: suppose first that $|\mathcal{X}| > \delta^{-\gamma}$, where $\gamma$ was defined in \eqref{eq:parameters}. Recall that 
$$(\nu \times \sigma)(\cup \mathcal{G}) \gtrsim \delta^{8\epsilon}.$$
In particular, since $\sigma(\R^{2}) \leq 1$, there exists $\theta_{0} \in \mathcal{D}_{\delta}(\spt \sigma)$ such that 
\begin{displaymath} \nu(\cup \{p \in \mathcal{X} : (p,\theta_{0}) \in \mathcal{G}\}) \gtrsim \delta^{8\epsilon}, \end{displaymath} 
Since $\nu$ is the uniform measure,
$$|\{p \in \mathcal{X} : (p,\theta_0) \in \mathcal{G}\}| \gtrsim \delta^{8\epsilon}|\mathcal{X}| \geq \delta^{-\gamma + 8\epsilon}.$$ 
It follows that
\begin{displaymath} |E|_{\delta} = |\{x + y : (x,y) \in G\}|_{\delta} \gtrsim |\{p\in \mathcal{X} : (p,\theta_{0}) \in \mathcal{G}\}|_{\delta} \geq \delta^{-\gamma + 8\epsilon} \stackrel{\eqref{form21}}{\geq} \delta^{-\alpha - \epsilon}. \end{displaymath}

For the remaining case, if $|\mathcal{X}| \leq \delta^{-\gamma}$, we are in a position to apply Proposition \ref{prop1} with parameters $\gamma,\beta,D$: for $\delta > 0$ small enough,
\begin{displaymath} |E|_{\delta} = |\{x + y : (x,y) \in G\}|_{\delta} \gtrsim |\{x + y : (x,y) \in \cup \mathcal{G}\}|_{\delta} \geq \delta^{-\epsilon_{0}}|\mathcal{X}| \stackrel{\eqref{form19}}{\geq} \delta^{-\alpha - \epsilon_{0} + 18\epsilon}. \end{displaymath}
Since $\epsilon_{0} - 18\epsilon \geq \epsilon$ by \eqref{form21}, we we have shown that $|E|_{\delta} \geq \delta^{-\alpha - \epsilon}$ in all cases. The proof is complete.  \end{proof}

We proceed to deduce Proposition \ref{prop1} from Lemma \ref{lemma1}.

\begin{proof}[Proof of Proposition \ref{prop1} assuming Lemma \ref{lemma1}] We start by fixing parameters. As in the previous proof, put 
$$\gamma := \tfrac{1}{2}(\alpha + 2) \in (\alpha,2).$$ 
Let $\epsilon_{0} := \epsilon_{0}(\gamma,\beta,D) > 0$ be the parameter given by Lemma \ref{lemma1}. Let $J = J(\alpha) \in \N$ be so large that
\begin{displaymath} \gamma (1 - 2^{-J - 1}) > \alpha. \end{displaymath}
Assume that $\delta > 0$ is so small that even $\delta^{2^{-J}}$ smaller than the scale threshold for Lemma \ref{lemma1} with parameters $\gamma,\beta$, and $T$. Finally, write $\eta := \eta(\alpha) := 2^{-J}$, and let $\epsilon = \epsilon(\alpha,\beta,D) > 0$ be so small that
\begin{equation}\label{form24} 4\epsilon/\eta < \epsilon_{0}. \end{equation}

Suppose now that $\mathcal{X},\sigma,\mathcal{G}$ satisfy the hypotheses of Proposition \ref{prop1} with parameter $\epsilon$. We show that 
$$|\{x + y : (x,y) \in \cup \mathcal{G}\}|_{\delta} \geq \delta^{- \epsilon}|\mathcal{X}|.$$

Our first goal is to prove that there exists a scale $\Delta = \delta^{2^{-j}}$, $j \in \{0,\ldots,J\}$, with the property
\begin{equation}\label{form22} |\mathcal{X} \cap Q|_{\Delta} \leq \Delta^{-\gamma/2}, \qquad Q \in \mathcal{D}_{\sqrt{\Delta}}(\mathcal{X}). \end{equation} 
Once this has been established, the idea is to complete the proof of Proposition \ref{prop1} by applying Lemma \ref{lemma1} at scale $\Delta$.

Assume to reach a contradiction that \eqref{form22} fails for all the scales $\Delta = \delta^{2^{-j}}$ for $j \in \{0,\ldots,J\}$. Since \eqref{form22} fails for $j = 0$, we may first deduce that
\begin{displaymath} |\mathcal{X}| \geq |\mathcal{X} \cap Q|_{\delta} \geq \delta^{-\gamma/2}, \qquad Q \in \mathcal{D}_{\sqrt{\delta}}(\mathcal{X}). \end{displaymath}
Next, since \eqref{form22} fails for $j = 1$, we may also deduce that
\begin{displaymath} |\mathcal{X}|_{\sqrt{\delta}} \geq |\mathcal{X} \cap Q|_{\sqrt{\delta}} \geq \delta^{-\gamma/4}, \qquad Q \in \mathcal{D}_{\delta^{1/4}}(\mathcal{X}). \end{displaymath}
Combining this with the previous inequality we find $|\mathcal{X}| \geq \delta^{-\gamma/4 - \gamma/2} = \delta^{-(3/4)\gamma}$. Continuing this way, if \eqref{form22} fails for all $j \in \{0,\ldots,J\}$, we may deduce 
\begin{displaymath} \delta^{-\alpha} \stackrel{\eqref{eq:sizeX}}{\geq} |\mathcal{X}| \geq \delta^{-\gamma  (1 - 2^{-J - 1})}. \end{displaymath}
Since $\gamma(1 - 2^{-J - 1}) > \alpha$, we reach a contradiction.

Recall that $\eta = 2^{-J}$, and let $\Delta = \delta^{2^{-j}} \in [\delta,\delta^{\eta}]$ be the scale we located just above; thus \eqref{form22} holds. We now plan to apply Lemma \ref{lemma1} at scale $\Delta$. For this purpose, write $\bar{\mathcal{X}} := \mathcal{D}_{\Delta}(\mathcal{X})$, and let $\bar{\nu}$ be the uniform probability on $\bar{X} := \cup \bar{\mathcal{X}}$. We will denote elements of $\bar{\mathcal{X}}$ by $Q$ and elements of $\mathcal{D}_{\Delta}(\spt \sigma)$ by $\Theta$. In the sequel we will use the following fact without further remark: thanks to the uniformity of $\mathcal{X}$,
\begin{displaymath} \nu(Q) = |\mathcal{X}|^{-1}|\mathcal{X} \cap Q| = |\mathcal{X}|_{\Delta}^{-1} = \bar{\nu}(Q), \qquad Q \in \bar{\mathcal{X}}. \end{displaymath}

Recall that $\mathcal{G} \subset \mathcal{X} \times \mathcal{D}_{\delta}(\spt \sigma)$ satisfies 
$$(\nu \times \sigma)(\cup \mathcal{G}) \geq \delta^{\epsilon}$$ 
by hypothesis. We produce a new subset $\bar{\mathcal{G}} \subset \bar{\mathcal{X}} \times \mathcal{D}_{\Delta}(\spt \sigma)$ such that 
$$(\bar{\nu} \times \sigma)(\cup \bar{\mathcal{G}}) \geq  \Delta^{\epsilon}.$$
For $(x,y_{0})\in \bar{\mathcal{X}}\times \mathcal{D}_{\Delta}(\spt \sigma)$ let $p(x),\theta(y_{0}) \in \mathcal{D}_{\delta}(\R^{2})$ be the unique $\delta$-squares containing them, respectively.  We declare that $(Q,\Theta) \in \bar{\mathcal{G}}$ if there exists an element $y_{0} \in \spt \sigma \cap \Theta$ such that
\begin{equation}\label{form6} \nu(\{x \in Q : (p(x),\theta(y_{0})) \in \mathcal{G}\}) \geq \delta^{2\epsilon} \nu(Q). \end{equation}
 We claim that 
 $$(\bar{\nu} \times \sigma)(\cup \bar{\mathcal{G}}) = (\nu \times \sigma)(\cup \bar{\mathcal{G}}) \geq \delta^{2\epsilon}.$$ 
To see this, using that 
$$(\nu \times \sigma)(\cup \mathcal{G} \cap (Q \times \Theta)) \leq \nu(Q)\sigma(\Theta),$$
we have:
\begin{equation}\label{form23} (\nu \times \sigma)(\cup \mathcal{G}) \leq \sum_{(Q,\Theta) \in \bar{\mathcal{G}}} \nu(Q)\sigma(\Theta) + \sum_{(Q,\Theta) \notin \bar{\mathcal{G}}} (\nu \times \sigma)(\cup \mathcal{G} \cap (Q \times \Theta)). \end{equation} 
The first sum equals $(\nu \times \sigma)(\cup \bar{\mathcal{G}})$. In the second sum, the fact that $(Q,\Theta) \notin \bar{\mathcal{G}}$ yields
\begin{displaymath} (\nu \times \sigma)(\cup \mathcal{G} \cap (Q \times \Theta)) = \int_{\Theta} \nu(\{x \in Q : (p(x),\theta(y)) \in \mathcal{G}\}) \, d\sigma(y) < \delta^{2\epsilon}\nu(Q)\sigma(\Theta).  \end{displaymath} 
So, the second sum in \eqref{form23} is $<  \delta^{2\epsilon}$. Therefore, the first sum is $\gtrsim  \delta^{2\epsilon}$, as claimed.

Let us recap the achievements so far. By \eqref{form22}, writing $\Delta = 2^{-\bar{m}T}$, we know that $\bar{\mathcal{X}}$ is a $\{2^{-jT}\}_{j = 1}^{\bar{m}}$-uniform set satisfying 
\begin{displaymath} |\bar{\mathcal{X}} \cap Q| \leq \Delta^{-\gamma/2}, \qquad Q \in \mathcal{D}_{\sqrt{\Delta}}(\bar{\mathcal{X}}). \end{displaymath} 
Moreover, $\bar{\nu}$ is the uniform measure on $\cup \bar{\mathcal{X}}$, and $\bar{\mathcal{G}} \subset \bar{\mathcal{X}} \times \mathcal{D}_{\Delta}(\spt \sigma)$ is a set satisfying
\begin{displaymath} (\bar{\nu} \times \sigma)(\cup \bar{\mathcal{G}}) \gtrsim \delta^{2\epsilon}  \geq  \Delta^{2\epsilon/\eta} \stackrel{\eqref{form24}}{\geq} \Delta^{\epsilon_{0}}. \end{displaymath}

Thus, applying Lemma \ref{lemma1} with the parameter  $\epsilon_{0} = \epsilon_{0}(\gamma,\beta,D) > 0$, we have
\begin{equation}\label{form7} |\{\mathbf{x} + \mathbf{y} : (\mathbf{x},\mathbf{y}) \in \cup \bar{\mathcal{G}}\}|_{\Delta} \geq \Delta^{-\epsilon_{0}}|\bar{\mathcal{X}}|. \end{equation} 
We aim to deduce \eqref{form5}. First claim that
\begin{displaymath} |\{x + y : (x,y) \in \cup \mathcal{G}\}|_{\delta} \gtrsim  |\{\mathbf{x} + \mathbf{y} : (\mathbf{x},\mathbf{y}) \in \cup \bar{\mathcal{G}}\}|_{\Delta} \cdot \min_{(Q,\Theta) \in \bar{\mathcal{G}}} |\{x + y : (x,y) \in \cup \mathcal{G} \cap (Q \times \Theta)\}|_{\delta} \end{displaymath} 
To see this, let 
$$\mathbf{Q}_{1},\ldots,\mathbf{Q}_{N} \in \mathcal{D}_{\Delta}(\{\mathbf{x} + \mathbf{y} : (\mathbf{x},\mathbf{y}) \in \cup \bar{\mathcal{G}}\})$$ 
be a maximal $(10\Delta)$-separated set. Thus, 
$$N \sim  |\{\mathbf{x} + \mathbf{y} : (\mathbf{x},\mathbf{y}) \in \cup \bar{\mathcal{G}}\}|_{\Delta}.$$
Now, for each $j \in \{1,\ldots,N\}$, we may fix $(Q_{j},\Theta_{j}) \in \bar{\mathcal{G}}$, and a pair $(\mathbf{x}_{j},\mathbf{y}_{j}) \in Q_{j} \times \Theta_{j}$, such that $\mathbf{x}_{j} + \mathbf{y}_{j} \in \mathbf{Q}_{j}$. Since both $Q_{j},\Theta_{j}$ are $\Delta$-squares,
\begin{displaymath} \{x + y : (x,y) \in \cup \mathcal{G} \cap (Q_{j} \times \Theta_{j})\} \subset B(\mathbf{x}_{j} + \mathbf{y}_{j},5\Delta). \end{displaymath}
Since the squares $\mathbf{Q}_{1},\ldots,\mathbf{Q}_{N}$ are $(10\Delta)$-separated, it follows that the sets
\begin{displaymath} \{x + y : (x,y) \in \cup \mathcal{G} \cap (Q_{j} \times \Theta_{j})\}, \qquad 1 \leq j \leq N, \end{displaymath}
are disjoint, and this gives the claim.

The factor $|\{\mathbf{x} + \mathbf{y} : (\mathbf{x},\mathbf{y}) \in \cup \bar{\mathcal{G}}\}|_{\Delta}$ is lower bounded by \eqref{form7}. To estimate the second factor, recall that if $(Q,\Theta) \in \bar{\mathcal{G}}$, then there is at least one element $y_{0} \in \spt \sigma \cap \Theta$ such that \eqref{form6} holds. Therefore, recalling also that $\nu(\cup \mathcal{A}) = |\mathcal{A} \cap \mathcal{X}|/|\mathcal{X}|$ for all $\mathcal{A} \subset \mathcal{D}_{\delta}(\R^{2})$,
\begin{align*} |\{x + y : (x,y) \in \cup \mathcal{G} \cap (Q \times \Theta)\}|_{\delta} & \geq |\{x + y_{0} : (x,y_{0}) \in \cup \mathcal{G} \cap (Q \times \Theta)\}|_{\delta}\\
& = |\{x \in Q : (x,y_{0}) \in \cup \mathcal{G}\}|_{\delta} \geq \delta^{2\epsilon} |\mathcal{X} \cap Q|. \end{align*} 
Therefore,
\begin{displaymath} |\{x + y : (x,y) \in \cup \mathcal{G} \cap (Q \times \Theta)\}|_{\delta} \gtrsim \Delta^{-\epsilon_{0}} \delta^{2\epsilon}|\bar{\mathcal{X}}| \cdot \min_{Q \in \bar{\mathcal{X}}} |\mathcal{X} \cap Q| = \Delta^{-\epsilon_{0}}\delta^{2\epsilon}|\mathcal{X}| \end{displaymath}
by the uniformity of $\mathcal{X}$. Recall finally from \eqref{form24}, and $\Delta \leq \delta^{\eta}$, that $\delta^{2\epsilon} \geq \Delta^{\epsilon_{0}/2}$. So, the previous displayed inequality implies
\begin{displaymath} |\{x + y : (x,y) \in \cup \mathcal{G} \cap (Q \times \Theta)\}|_{\delta} \gtrsim \Delta^{-\epsilon_{0}/2}|\mathcal{X}| \geq \delta^{-\eta \epsilon_{0}/2}|\mathcal{X}| \geq \delta^{-2\epsilon}|\mathcal{X}|. \end{displaymath} 
This completes the proof of Proposition \ref{prop1}. \end{proof}

\subsection{Proof of Lemma \ref{lemma1}} \label{Section:proof of Lemma1}
By the arguments laid out in the previous Section, to prove Proposition \ref{prop:main} it suffices to prove Lemma \ref{lemma1}. This is the purpose of this Section.
\subsubsection{Choice of parameters and an assumption towards a contradiction} \label{section:parameters}
We start by fixing parameters. Let $\mathcal{X},\sigma,\mathcal{G}$ be as in the statement of Lemma \ref{lemma1}, and recall that the Borel probability measure $\sigma$ is assumed to be $(D,\beta)$-uniformly perfect. Let $A \geq 1$ be an absolute constant to be determined a little later. For 
\begin{equation}\label{choiceEta} \eta := (2 - \alpha)/4, \end{equation}
let $\epsilon_{0} := \epsilon_{0}(A^{2}D,\beta,\eta) > 0$ and $\delta_{0} := \delta_{0}(A^{2}D,\beta,\eta,(AD)^{-1}) > 0$ be the constants given by Proposition \ref{RSProp}. Assume that $\epsilon \leq \epsilon_{0}/60$, and additionally $\epsilon > 0$ is so small that, via \eqref{choiceEta},
\begin{equation}\label{choiceEpsilon} 1 - \eta - \tfrac{50\epsilon}{s} > \tfrac{\alpha}{2}, \end{equation} 
where $s = -\log \beta/\log D$ is the Frostman exponent of $\sigma$, recall Lemma \ref{frostmanLemma}. Let $\delta > 0$ be so small that $\sqrt{\delta} \leq \delta_{0}$.

Write $\Delta := \sqrt{\delta}$ and $X := \cup \mathcal{X} = \spt \nu$, where we recall that $\nu$ is the uniform probability measure on $X$. Let $c > 0$ an absolute constant so small that the following holds: if $\theta$ is a disc of radius $c\Delta$, then $\mathbb{P} \cap \theta$ is contained in a rectangle of dimensions $\delta \times \Delta$. This rectangle will be denoted $R(\theta)$ for the remainder of the proof. Finally, we can choose our $A$ as $A := 10/c$.

Assume, towards a contradiction, that
\begin{equation}\label{form9} |\{x + y : (x,y) \in \cup \mathcal{G}\}|_{\delta} \leq \delta^{-\epsilon}|\mathcal{X}|. \end{equation}
In broad terms, our strategy is to show that \eqref{form9} invalidates our local growth assumption assumption \eqref{eq:sizeXloc}. This is an inverse Theorem-like strategy, in the sense that lack of growth of a sum-set can only be explained by one of the ambient sets begin already quite large. And, indeed, the main tool in the proof will be Proposition \ref{RSProp}, though the precise way in which it is applied is quite subtle.

We begin with some initial definitions and constructions, that will accompany us throughout the proof.

\subsubsection{Preliminary definitions and constructions}
We begin with the following simple Claim.
\begin{claim} \label{claim:cover}
There exists a cover $\mathcal{B}_{\Delta}$  of $\spt \sigma$ such that:
\begin{enumerate}
\item Every $B \in \mathcal{B}_{\Delta}$ is a ball of radius $c\Delta$.

\item It has bounded overlap.

\item It satisfies
\begin{equation}\label{form35} \diam(\theta \cap \spt \sigma) \geq c\Delta/D, \qquad \theta \in \mathcal{B}_{\Delta}. \end{equation}
\end{enumerate}
\end{claim}
\begin{proof}
Let $\Sigma \subset \spt \sigma$ be a maximally $(c\Delta/2)$-separated subset (recall that $\sigma$ is compactly supported). Since $\sigma$ is $(D,\beta)$-uniformly perfect, we have 
$$0<\sigma(B(y,c\Delta/D)) < \sigma(B(y,c\Delta)),\quad  y \in \spt \sigma.$$
This implies that
$$\diam(\spt \sigma \cap B(y,c\Delta)) \geq c\Delta/D,\quad y \in \spt \sigma.$$
So,  
$$\mathcal{B}_{\Delta} := \{B(y,c\Delta) : y \in \Sigma\}$$
is the desired collection of balls. 
\end{proof}

For $\theta \in \mathcal{B}_{\Delta}$, we write $\sigma_{\theta} := \sigma|_{\theta}$ and $\sigma_{A\theta} := \sigma|_{A\theta}$, where we recall that $A = 10/c$.  Recall that $A\theta$ denotes the disc concentric to $\theta$ with side $10\Delta$. We proceed to single out a sub-collection of these balls that are more relevant to us. 
\begin{definition} \label{def:Theta}
We say that $\theta \in \mathcal{B}_{\Delta}$ is \emph{good}, denoted $\theta \in \Theta$, if
\begin{equation}\label{form8} (\nu \times \sigma_{\theta})(\cup \mathcal{G}) \geq \delta^{2\epsilon} \|\sigma_{A\theta}\|,\text{ where } \|\sigma_{A\theta}\| := \sigma_{A\theta}(\R^{2}). \end{equation} 
\end{definition}
\begin{claim} \label{claim:only use}
If $\delta>0$ is   sufficiently small, then
$$\sum_{\theta \in \Theta} \sigma(\theta) \geq \delta^{2\epsilon}.$$
\end{claim} 
\begin{proof}
By assumption, we have
\begin{displaymath} \delta^{\epsilon} \leq (\nu \times \sigma)(\cup \mathcal{G}) \leq \sum_{\theta \in \Theta} \sigma(\theta) + \delta^{2\epsilon} \sum_{\theta \notin \Theta} \|\sigma_{A\theta}\|. \end{displaymath} 
Thanks to the bounded overlap of the family $\{\spt \sigma_{A\theta}\}_{\theta \in \mathcal{B}_{\Delta}}$, and since $\sigma$ is a probability measure,
$$\delta^{2\epsilon} \sum_{\theta \notin \Theta} \|\sigma_{A\theta}\| \lesssim \delta^{2\epsilon},$$ 
and the claim follows.
\end{proof}

Write $\mathcal{X}_{\Delta} := \mathcal{D}_{\Delta}(\mathcal{X})$. For $\theta \in \Theta$, we define a set of "good" squares $\mathcal{G}_{\theta} \subset \mathcal{X}_{\Delta}$ as follows.
\begin{definition} \label{def:G theta}
Fix $\theta \in \Theta$.  We declare that $Q \in \mathcal{X}_{\Delta}$ is an element of $\mathcal{G}_{\theta}$ if
\begin{itemize}
\item[(G1) \phantomsection \label{G1}] $(\nu \times \sigma_{\theta})(\cup \mathcal{G}) \geq \delta^{3\epsilon}\nu(Q)\|\sigma_{A\theta}\|$,
\item[(G2) \phantomsection \label{G2}] $|\{x + y : (x,y) \in \cup \mathcal{G} \cap (Q \times \theta)\}|_{\delta} \leq \delta^{-5\epsilon}|\mathcal{X} \cap Q|$. 
\end{itemize}
\end{definition}

We proceed to give non-trivial estimates on the size of the set of  $\mathcal{G}_{\theta}$.
\begin{claim}\label{claim:only use 2}
If $\delta > 0$ is small enough in terms of $\epsilon$, then:
\begin{enumerate}
\item We have
\begin{equation}\label{form26} |\{Q \in \mathcal{X}_{\Delta} \text{ satisfies \nref{G1}}\}| \geq \delta^{3\epsilon} |\mathcal{X}_{\Delta}|. \end{equation}

\item Let $\mathcal{B}_{\theta} \subset \mathcal{X}_{\Delta}$ be the subset failing \nref{G2}, that is,
\begin{displaymath} \mathcal{B}_{\theta} := \{Q \in \mathcal{X}_{\Delta} : |\{x + y : (x,y) \in \cup \mathcal{G} \cap (Q \times \theta)\}| > \delta^{-5\epsilon}|\mathcal{X} \cap Q|\}. \end{displaymath}
Then ,
\begin{displaymath} |\mathcal{B}_{\theta}| \leq \tfrac{1}{2}\delta^{3\epsilon}|\mathcal{X}_{\Delta}| \stackrel{\eqref{form26}}{\leq} \tfrac{1}{2}|\{Q \in \mathcal{X}_{\Delta} \text{ satisfies \nref{G1}}\}|. \end{displaymath}
\end{enumerate}
In particular,
\begin{equation}\label{form11} |\mathcal{G}_{\theta}| \geq \delta^{4\epsilon}|\mathcal{X}_{\Delta}|, \qquad \theta \in \Theta. \end{equation}
\end{claim}

Morally, the idea is that "nearly all" squares in $\mathcal{X}_{\Delta}$ satisfy \nref{G2}, and positively many squares satisfy \nref{G1}. Therefore positively many squares satisfy both \nref{G1}-\nref{G2}. 
\begin{proof}
Part (1) follows from \eqref{form8} by estimating
\begin{displaymath} \delta^{2\epsilon}\|\sigma_{A\theta}\| \leq (\nu \times \sigma_{\theta})(\cup \mathcal{G}) \leq \sum_{Q \text{ satisfies \nref{G1}}} \nu(Q)\|\sigma_{\theta}\| + \sum_{Q \text{ fails \nref{G1}}} \delta^{3\epsilon}\nu(Q)\|\sigma_{A\theta}\|. \end{displaymath}
The second term is $\leq \delta^{3\epsilon}\|\sigma_{\theta}\|$. So, 
$$\nu(\cup \{Q \in \mathcal{X}_{\Delta} \text{ satisfies \nref{G1}}\}) \geq \delta^{3\epsilon}.$$ 
Now \eqref{form26} follows from the uniformity of $\mathcal{X}$.

For Part (2), first use \eqref{form9} to deduce
\begin{equation}\label{form10} |\{x + y : (x,y) \in \cup \mathcal{G} \cap (X \times \theta)\}|_{\delta} \leq |\{x + y : (x,y) \in \cup \mathcal{G}\}|_{\delta} \leq \delta^{-\epsilon}|X|. \end{equation} 
Since $\diam(\theta) \leq \Delta$, the sets 
\begin{displaymath} \{x + y : (x,y) \in \cup \mathcal{G} \cap (Q \times \theta)\}, \qquad Q \in \mathcal{X}_{\Delta}, \end{displaymath}
have bounded overlap. Therefore,
\begin{equation}\label{form27} |\{x + y : (x,y) \in \cup \mathcal{G} \cap (X \times \theta)\}|_{\delta} \gtrsim \sum_{Q \in \mathcal{X}_{\Delta}} |\{x + y : (x,y) \in \cup \mathcal{G} \cap (Q \times \theta)\}|_{\delta}. \end{equation}
 Let $\mathcal{B}_{\theta} \subset \mathcal{X}_{\Delta}$ be the subset failing \nref{G2}, defined in Part (2). 
We deduce from \eqref{form10}, \eqref{form27}, and the uniformity of $\mathcal{X}$, 
\begin{displaymath} \delta^{-\epsilon}|\mathcal{X}| \geq |\{x + y : (x,y) \in \cup \mathcal{G} \cap (X \times \theta)\}|_{\delta} \gtrsim \delta^{-5\epsilon} |\mathcal{B}_{\theta}| \cdot |\mathcal{X} \cap Q| = \delta^{-5\epsilon}|\mathcal{X}| \cdot \frac{|\mathcal{B}_{\theta}|}{|\mathcal{X}_{\Delta}|}. \end{displaymath}
Therefore $|\mathcal{B}_{\theta}| \lesssim \delta^{4\epsilon}|\mathcal{X}_{\Delta}|$. In particular, for $\delta > 0$ small enough,
\begin{displaymath} |\mathcal{B}_{\theta}| \leq \tfrac{1}{2}\delta^{3\epsilon}|\mathcal{X}_{\Delta}| \stackrel{\eqref{form26}}{\leq} \tfrac{1}{2}|\{Q \in \mathcal{X}_{\Delta} \text{ satisfies \nref{G1}}\}|. \end{displaymath}
Therefore, for $\delta > 0$ small enough, at least $\tfrac{1}{2}\delta^{3\epsilon}|\mathcal{X}_{\Delta}|$ squares in $\mathcal{X}_{\Delta}$ satisfy both \nref{G1}-\nref{G2}, and this finally yields \eqref{form11}. 
\end{proof}

Claims \ref{claim:only use} and \ref{claim:only use 2} have the following consequence:
\begin{cor} \label{coro:Q}
If $\delta > 0$ small enough in terms of $\epsilon$, there exists a square $\mathbf{Q} \in \mathcal{D}_{\Delta}(X)$,  such that
\begin{equation}\label{form12} \sigma(\cup \{\theta \in \Theta : \mathbf{Q} \in \mathcal{G}_{\theta}\}) \geq \delta^{7\epsilon}. \end{equation}
\end{cor}
\begin{proof}
Combining Claim \ref{claim:only use} and Claim \ref{claim:only use 2} we have
\begin{displaymath} \delta^{6\epsilon}|\mathcal{X}_{\Delta}| \leq \sum_{\theta \in \Theta} |\mathcal{G}_{\theta}|\sigma(\theta) \lesssim \sum_{Q \in \mathcal{X}_{\Delta}} \sigma(\cup \{\theta : \theta \in \Theta \text{ and } Q \in \mathcal{G}_{\theta}\}). \end{displaymath}
\end{proof}

We fix $\mathbf{Q}$ for the reminder of the proof. We proceed to investigate the structure of $X \cap \mathbf{Q}$
\subsubsection{Slices and projections of our chosen cube} 

Recall that the Collection $\Theta \subset  \mathcal{B}_{\Delta}$ was defined in Definition \ref{def:Theta} (and the cover $\mathcal{B}_{\Delta}$ of $\spt \sigma$ was constructed in Claim \ref{claim:cover}), and the collection $\mathcal{G}_{\theta} \subset \mathcal{X}_{\Delta}$ for $\theta \in \Theta$ was defined in Definition \ref{def:G theta}.

\begin{remark}\label{rem6}  Recall from  Section \ref{section:parameters} that $R(\theta)$ is a $(\delta \times \Delta)$-rectangle containing $\mathbb{P} \cap c\theta$, where $c\theta = B(z_{\theta},c\Delta)$ is a disc centred at $z_{\theta} \in \mathbb{P}$. We may write $z_{\theta} = (x_{\theta},\varphi(x_{\theta}))$ for $x_{\theta} \in [-1,1]$. Now, the longer side of $R(\theta)$ is (or can be taken to be) parallel to the tangent line of $\mathbb{P}$ at $z_{\theta}$, and this line is a translate of $\ell_{\theta} := \mathrm{span}(1,\varphi'(x_{\theta}))$. We define $\pi_{\theta}$ as the orthogonal projection to the line $\ell_{\theta}^{\perp}$. So, $\pi_{\theta}$ is the orthogonal projection "along" the rectangle $R(\theta)$. \end{remark} 
We are now ready to state the main result of this Section. Recall that $\eta$ was defined in \eqref{choiceEta}.
\begin{claim}\label{c1} For every $\theta \in \Theta$ such that $\mathbf{Q} \in \mathcal{G}_{\theta}$, there exists a set $\mathcal{X}_{\theta} \subset \mathcal{X} \cap \mathbf{Q}$ such that:
\begin{enumerate}
\item  $ \delta^{5\epsilon} |\mathcal{X} \cap \mathbf{Q}| \leq |\mathcal{X}_{\theta}|$; and

\item 
  $|\pi_{\theta}(\cup \mathcal{X}_{\theta})|_{\delta} \lesssim \Delta^{1 - \eta} |\mathcal{X}_{\theta}|$.
\end{enumerate}
  \end{claim}

\begin{proof}[Proof of Claim \ref{c1} Part (1)] Fix $\theta \in \Theta$ with $\mathbf{Q} \in \mathcal{G}_{\theta}$, and recall that $\mathbf{Q}$ satisfies \nref{G1}-\nref{G2}. Let $\mathcal{R}_{0}$ be a minimal cover of $\mathbf{Q}$ by disjoint rectangles of dimensions $\delta \times \Delta$ with longer side parallel to $R(\theta)$. For $R \in \mathcal{R}_{0}$, write
\begin{displaymath} R \cap \mathbf{Q} := \cup \mathcal{D}_{\delta}(\mathbf{Q} \cap R) \quad \text{and} \quad \nu_{R} := \nu|_{R \cap \mathbf{Q}}. \end{displaymath}
Thus, $R \cap \mathbf{Q}$ is a union of $\delta$-squares contained in $2R$, and the sets $R \cap \mathbf{Q}$, $R \in \mathcal{R}_{0}$, have bounded overlap. We only care about those rectangles $R \in \mathcal{R}_{0}$ such that
\begin{equation}\label{form17} (\nu_{R} \times \sigma_{\theta})(\cup \mathcal{G}) \geq \delta^{4\epsilon} \|\nu_{R}\|\|\sigma_{A\theta}\|. \end{equation}
We denote these rectangles $\mathcal{R}$, and we define
\begin{displaymath} \bar{\mathcal{X}} \cap \mathbf{Q} := \mathcal{X} \cap \bigcup_{R \in \mathcal{R}} (R \cap \mathbf{Q}) := \{p \in \mathcal{X} \cap \mathbf{Q} : p \cap R \neq \emptyset \text{ for some } R \in \mathcal{R}\}. \end{displaymath}
Writing $\bar{X} := \cup \bar{\mathcal{X}}$ (thus $\bar{X}$ is a union of $\delta$-squares), we claim that 
\begin{equation}\label{form13} \nu(\bar{X} \cap \mathbf{Q}) \geq \delta^{4\epsilon} \nu(\mathbf{Q}), \text{ which is equivalent to } |\bar{\mathcal{X}} \cap \mathbf{Q}| \geq \delta^{4\epsilon} |\mathcal{X} \cap \mathbf{Q}|. \end{equation}
To prove \eqref{form13}, note that
\begin{displaymath} \sum_{R \in \mathcal{R}_{0} \, \setminus \, \mathcal{R}} (\nu_{R} \times \sigma_{\theta})(\cup \mathcal{G}) \leq \delta^{4\epsilon} \|\sigma_{A\theta}\| \sum_{R \in \mathcal{R}_{0}} \|\nu_{R}\| \lesssim \delta^{4\epsilon}\nu(\mathbf{Q})\|\sigma_{A\theta}\|. \end{displaymath} 
So, by \nref{G1}, and provided that $\delta > 0$ is small enough,
\begin{displaymath} \nu(\bar{X} \cap \mathbf{Q})\|\sigma_{\theta}\| \geq \sum_{R \in \mathcal{R}} (\nu_{R} \times \sigma_{\theta})(\cup \mathcal{G}) \stackrel{\textup{\nref{G1}}}{\geq} \tfrac{1}{2}\delta^{3\epsilon} \nu(\mathbf{Q})\|\sigma_{A\theta}\|. \end{displaymath}
This yields \eqref{form13}. 

We reduce $\mathcal{R}$ a little further. By dyadic pigeonholing, we may select a subset $\mathcal{R}' \subset \mathcal{R}$ such that $R \mapsto |R \cap \mathbf{Q} \cap \mathcal{X}|$ (or, equivalently, $R \mapsto \nu(R \cap \mathbf{Q})$) is roughly constant on $\mathcal{R}'$, and still
\begin{displaymath} \nu\Big(\bigcup_{R \in \mathcal{R}'} (R \cap \mathbf{Q})\Big) \geq \delta^{\epsilon} \nu(\bar{X} \cap \mathbf{Q}). \end{displaymath}
Here $R \cap \mathbf{Q} \cap \mathcal{X} := \{p \in \mathcal{X} \cap \mathbf{Q} : p \cap R \neq \emptyset\}$. Replacing $\delta^{4\epsilon}$ by $\delta^{5\epsilon}$ in \eqref{form13}, we may assume that the family $\mathcal{R}$ had the constancy property above to begin with, say
\begin{equation}\label{form14} |R \cap \mathbf{Q} \cap \mathcal{X}| \in [m,2m], \, R \in \mathcal{R}, \text{ where } m \in [1,2\Delta^{-1}] \text{ is independent of } R \in \mathcal{R}.
  \end{equation}
  
Notice that the sets 
\begin{equation}\label{form28} \{x + y : (x,y) \in R \times R(\theta)\}, \qquad R \in \mathcal{R}, \end{equation}
have bounded overlap.   Indeed, this follows since the $(\delta \times \Delta)$-rectangles $R \in \mathcal{R}$ are parallel to the $(\delta \times \Delta)$-rectangle $R(\theta)$, and intersect the fixed $\Delta$-square $\mathbf{Q}$ (think of the case where $R,R(\theta)$ are parallel to the $x_{1}$-axis; here the $x_{2}$-coordinates of the sets in \eqref{form28} have bounded overlap).  Since $\spt \sigma \cap \theta$ is contained in $R(\theta)$ (and $\mathcal{G} \subset \mathcal{X} \times \mathcal{D}_{\delta}(\spt \sigma)$), consequently also the following sets have bounded overlap:
\begin{displaymath} \{x + y : (x,y) \in \cup \mathcal{G} \cap ([R \cap \mathbf{Q}] \times \theta)\}, \qquad R \in \mathcal{R}. \end{displaymath}
Therefore,
\begin{align} \sum_{R \in \mathcal{R}} & |\{x + y : (x,y) \in \cup \mathcal{G} \cap ([R \cap \mathbf{Q}] \times \theta)\}|_{\delta} \notag\\
& \lesssim |\{x + y : (x,y) \in \cup \mathcal{G} \cap (\mathbf{Q} \times \theta)\}|_{\delta} \notag\\
&\label{form15} \stackrel{\textup{\nref{G2}}}{\leq} \delta^{-5\epsilon}|\mathcal{X} \cap \mathbf{Q}| \stackrel{\eqref{form13}}{\leq} \delta^{-10\epsilon}|\bar{\mathcal{X}} \cap \mathbf{Q}| \sim \delta^{-10\epsilon}|\mathcal{R}| \cdot m. \end{align}
Finally, let $\mathcal{R}_{\mathrm{good}} \subset \mathcal{R}$ consist of those rectangles $R \in \mathcal{R}$ such that
\begin{equation}\label{form18} |\{x + y : (x,y) \in \cup \mathcal{G} \cap ([R \cap \mathbf{Q}] \times \theta)\}|_{\delta} \leq \delta^{-11\epsilon}|R \cap \mathbf{Q} \cap \mathcal{X}| \sim \delta^{-11\epsilon} m. \end{equation}
We note that $\mathcal{R} \, \setminus \, \mathcal{R}_{\mathrm{good}}$ is rather small:
\begin{displaymath} |\mathcal{R} \, \setminus \, \mathcal{R}_{\mathrm{good}}| \cdot \delta^{-11\epsilon}m \lesssim \sum_{R \in \mathcal{R} \, \setminus \mathcal{R}_{\mathrm{good}}} |\{x + y : (x,y) \in \cup \mathcal{G} \cap ([R \cap \mathbf{Q}] \times \theta)\}| \stackrel{\eqref{form15}}{\leq} \delta^{-10\epsilon}|\mathcal{R}| \cdot m.  \end{displaymath} 
In particular $|\mathcal{R}_{\mathrm{good}}| \geq \tfrac{1}{2}|\mathcal{R}|$. Therefore, by the rough constancy of $R \mapsto |R \cap \mathbf{Q} \cap \mathcal{X}|$,
\begin{equation}\label{form16} \left| \mathcal{X}_{\theta} \right| \sim |\bar{\mathcal{X}} \cap \mathbf{Q}| \stackrel{\eqref{form13}}{\geq} \delta^{5\epsilon}|\mathcal{X} \cap \mathbf{Q}|, \end{equation} 
where $\mathcal{X}_{\theta} := \bigcup_{R \in \mathcal{R}_{\mathrm{good}}} (R \cap \mathbf{Q} \cap \mathcal{X})$. This proves Part (1).
\end{proof}

So far, we have dealt wit Part (1), and constructed the set $\mathcal{X}_{\theta}$ at the end of its proof.
 It remains to prove  Part (2), namely that 
 $$|\pi_{\theta}(\cup \mathcal{X}_{\theta})|_{\delta} \lesssim \Delta^{1 - \eta} |\mathcal{X}_{\theta}|.$$ 
 We will accomplish this by demonstrating that
\begin{equation}\label{form37} |R \cap \mathbf{Q} \cap \mathcal{X}| \gtrsim \Delta^{\eta - 1}, \qquad R \in \mathcal{R}_{\mathrm{good}}. \end{equation}
This suffices, since $\pi_{\theta}$ is the projection along the longer side of $R$, so $\pi_{\theta}$ maps all the squares in $R \cap \mathbf{Q} \cap \mathcal{X}$ inside a single interval of length $\sim \delta$. 

Fix $R \in \mathcal{R}_{\mathrm{good}} \subset \mathcal{R}$, and recall from \eqref{form17} and \eqref{form18} that 
\begin{itemize}
\item[(1) \phantomsection \label{1}] $(\nu_{R} \times \sigma_{\theta})(\cup \mathcal{G}) \geq \delta^{4\epsilon}\|\nu_{R}\|\|\sigma_{A\theta}\|$,
\item[(2) \phantomsection \label{2}] $ |\{x + y : (x,y) \in \cup \mathcal{G} \cap ([R \cap \mathbf{Q}] \times \theta)\}|_{\delta} \leq \delta^{-11\epsilon} |R \cap \mathbf{Q} \cap \mathcal{X}|$. 
\end{itemize}
With \nref{1}-\nref{2} in hand, the plan is to apply Proposition \ref{RSProp}. Let us sketch the idea first. Starting from the measures $\nu_{R}$ (defined in the beginning of the proof of Part (1)) and the restriction $\sigma_{\theta}$, we will construct $\Delta$-measures $\bar{\nu}$ and $\bar{\sigma}$ on $[0,1]$ such that:
\begin{enumerate}
\item $\bar{\nu}$ has (roughly) constant density; and
\item  $\bar{\sigma}$ is uniformly perfect with $\diam(\spt \bar{\sigma}) \gtrsim 1/D$. 
\end{enumerate} 
 Then we use $\mathcal{G}$ to construct a "fat" subset $G \subset \spt \bar{\nu} \times \bar{\sigma}$ such that 
 $$(\bar{\nu} \times \bar{\sigma})(G) \approx 1 \text{ and } |\{x + y : (x,y) \in G\}| \lessapprox |\mathrm{spt\,} \bar{\nu}| \sim m.$$ 
Applying Lemma \ref{lemma2}, and we find 
$$\|\bar{\nu} \ast \bar{\sigma}\|_{L^{2},\mathrm{Sh}} \approx \|\bar{\nu}\|_{L^{2},\mathrm{Sh}}.$$ 
It follows that $\bar{\nu}$ must violate condition (1) of Proposition \ref{RSProp}, and so  
$$\|\bar{\nu}\|_{L^{2},\mathrm{Sh}}^{2} \leq \Delta^{1 - \eta}.$$
This is \eqref{form37}.

We turn to the details. 

\begin{proof}[Proof of Claim \ref{c1} Part (2)] Recall that $R$ is a $(\Delta \times \delta)$-rectangle parallel to the $(\Delta \times \delta)$-rectangle $R(\theta)$. We start by defining two "almost" similarities $T_{1},T_{2}$. The measures $\bar{\nu},\bar{\sigma}$ will (almost) be defined as renormalised push-forwards under $T_{1},T_{2}$. 

First, let $T_{1}'$ be a similarity map taking $R$ to $[\tfrac{1}{4},\tfrac{3}{4}] \times [0,\tfrac{1}{2}\Delta]$; thus $T_{1}'$ can be written as
\begin{equation}\label{form29} T_{1}'(z) = \mathcal{O}( (2\Delta)^{-1}z) + z_{1}, \end{equation} 
where $\mathcal{O}=\mathcal{O}(R)$ is a rotation, and $z_{1}=z_1(R) \in \R^{2}$.
Then $T_{1}'$ maps the $\delta$-squares $p \in R \cap \mathbf{Q}$ to $\tfrac{1}{2}\Delta$-squares contained in $[\tfrac{1}{8},\tfrac{7}{8}] \times [-2\Delta,2\Delta]$. For each of these squares $T_{1} '(p)$, choose a ("nearest") point $x_{p} \in \Delta \Z \cap [0,1]$ such that 
$$\dist((x_{p},0),T_{1}'(p)) \lesssim \Delta.$$ 
Note that we work with the interval $[\tfrac{1}{4},\tfrac{3}{4}]$  to ensure that the points $x_{p}$ land in $[0,1]$.

Finally, let $T_{1} \colon \cup (R \cap \mathbf{Q}) \to \Delta \Z \cap [0,1]$ be the map which sends $p \in R \cap \mathbf{Q}$ entirely to the point $x_{p}$, thus $T_{1}(p) := \{x_{p}\}$. Let 
\begin{displaymath} \bar{\nu}_{R} := \|\nu_{R}\|^{-1}T_{1}(\nu_{R}). \end{displaymath}
\begin{remark}\label{rem4} Let consider for a moment  the measure $\bar{\nu}_{R}$.  Recall that $p \mapsto \nu(p)$ is constant on $\mathcal{X}$, and in particular on $R \cap \mathbf{Q} \cap \mathcal{X}$. On the other hand, the map $T_{1}$ only sends boundedly many squares $p \in R \cap \mathbf{Q}$ to a single point $x \in \Delta \Z$. So, the density of $\bar{\nu}_{R}$ is also roughly constant on $\spt \bar{\nu}_{R}$. Therefore, we may think of $\bar{\nu}_{R}$  roughly as the uniform probability measure on 
$$T_{1}(\cup (R \cap \mathbf{Q} \cap \mathcal{X})) \subset \Delta \Z \cap [0,1].$$ \end{remark}

We then proceed to define $\bar{\sigma}$. Recall again that $R(\theta)$ is a $(\Delta \times \delta)$-rectangle parallel to $R$. The map $T_{1}'$ therefore sends $R(\theta)$ to some $(\tfrac{1}{2} \times \tfrac{1}{2}\Delta)$-rectangle parallel to $[\tfrac{1}{2},\tfrac{3}{4}] \times [0,\tfrac{1}{2}\Delta]$. We choose $z_{2} \in \R^{2}$ (depending only on $\theta$) such that 
\begin{equation}\label{form41} T_{1}'(R(\theta)) - z_{2} = [\tfrac{1}{2},\tfrac{3}{4}] \times [0,\tfrac{1}{2}\Delta], \end{equation}
and then we define 
\begin{displaymath} T_{2}'(x) := T_{1}'(x) - z_{2}, \qquad x \in \R^{2}. \end{displaymath}
With this notation, $T_{2}'$ maps $\delta$-squares $q \in \mathcal{D}_{\delta}(\mathbb{P} \cap \theta) \subset \mathcal{D}_{\delta}(R(\theta))$ to $\tfrac{1}{2}\Delta$-squares contained in $[\tfrac{1}{8},\tfrac{7}{8}] \times [-2\Delta,2\Delta]$. Note also that $A\theta \cap \mathbb{P}$ is contained in some (absolute) enlargement of $R(\theta)$. So $T_{2}'$ maps the $\delta$-squares $q \in \mathcal{D}_{\delta}(\mathbb{P} \cap A\theta)$ to $[-C,C] \times [-C\Delta,C\Delta]$ for an absolute constant $C \geq 1$. For each $q \in \mathcal{D}_{\delta}(\mathbb{P} \cap A\theta)$, choose some point $y_{q} \in \Delta \Z \cap [-C,C]$ with
\begin{equation}\label{form33} \dist((y_{q},0),T_{2}'(q)) \lesssim \Delta. \end{equation}
Then, define $T_{2} \colon \cup \mathcal{D}_{\delta}(\mathbb{P} \cap A\theta) \to \Delta \Z$ with the same idea as $T_{1}$, by requiring 
\begin{displaymath} T_{2}(q) := \{y_{q}\}, \qquad q \in \mathcal{D}_{\delta}(\mathbb{P} \cap A\theta). \end{displaymath}
\begin{remark} We record for future reference that if $\delta$ (hence $\Delta$) is sufficiently small, then
\begin{equation}\label{form43} T_{2}(R(\theta)) \subset [0,1] \quad \text{and} \quad T_{2}'(A\theta) \supset B(0,3).  \end{equation}
The first inclusion follows from \eqref{form41}-\eqref{form33}. The second inclusion follows by recalling that $A\theta$ is a $10\Delta$-disc concentric with $\theta$, and noting that e.g. $(\tfrac{1}{2},0) \in T_{2}'(R(\theta)) \subset T_{2}'(A\theta)$.
\end{remark}

Noting that $\spt \sigma_{A\theta} \subset \mathbb{P} \cap A\theta$, we may now define the measures
\begin{displaymath} \bar{\Sigma} := \|\sigma_{A\theta}\|^{-1}T_{2}\sigma_{A\theta} \quad \text{and} \quad \bar{\sigma} := \bar{\Sigma}([-2,2])^{-1}\bar{\Sigma}|_{[-2,2]}. \end{displaymath}
Evidently $\bar{\sigma}$ is a $\Delta$-measure on $[-2,2]$. We first show that the support of $\bar{\sigma}$ has large diameter.
\begin{claim} We have $\diam(\spt \bar{\sigma}) \gtrsim D^{-1}$. \end{claim}

\begin{proof} Recall from \eqref{form41} that $T_{2}'(\spt \sigma_{\theta}) \subset T_{2}'(R(\theta)) \subset [\tfrac{1}{2},\tfrac{3}{4}] \times [0,\tfrac{1}{2}\Delta]$.  By \eqref{form35}
$$\diam(\spt \sigma_{\theta}) \gtrsim \Delta/D.$$
It follows that 
\begin{equation}\label{form42} \diam(T_{2}'(\spt \sigma_{A\theta}) \cap [\tfrac{1}{2},\tfrac{3}{4}] \times [0,\tfrac{1}{2}\Delta]) \geq \diam(T_{2}'(\spt \sigma_{\theta})) \gtrsim D^{-1}. \end{equation} 
Further, if $q \in \mathcal{D}_{\delta}(\spt \sigma_{A\theta})$, and $T_{2}'(q) \subset [\tfrac{1}{2},\tfrac{3}{4}] \times [0,\tfrac{1}{2}\Delta]$, then $\{y_{q}\} = T_{2}(q) \subset \spt \bar{\Sigma} \cap [0,1]$ thanks to \eqref{form33}. Therefore \eqref{form42} implies $\diam(\spt \bar{\Sigma}|_{[-2,2]}) \gtrsim D^{-1}$. \end{proof}

We proceed to show that $\bar{\sigma}$ is uniformly perfect. Recall from Definition \ref{def:uniformlyPerfect} that the uniform perfectness of $\Delta$-measures only requires the defining inequality to hold for radii $r \geq \Delta$.

\begin{claim}\label{claim3} The $\Delta$-measure $\bar{\sigma}$ is $(A^{2}D,\beta,[-2,2])$-uniformly perfect. \end{claim} 

\begin{proof} Note that $\sigma_{A\theta}$ is $(D,\beta,A\theta)$-uniformly perfect by Lemma \ref{lemma3}. Since $T_{2}'$ is a similarity, and $T_{2}'(A\theta) \supset B(0,3)$ by \eqref{form43},
\begin{displaymath} \bar{\Sigma}_{\R^{2}} := \|\sigma_{A\theta}\|^{-1}T_{2}'\sigma_{A\theta} \end{displaymath}
is $(D,\beta,B(0,3))$-uniformly perfect by Lemma \ref{lemma4}. Moreover, for $r \geq \Delta$ and $x \in \R$, \eqref{form33} implies that, as $A=10/c$, 
\begin{displaymath} \bar{\Sigma}(B(x,r)) \leq \bar{\Sigma}_{\R^{2}}(B((x,0),Ar)) \quad \text{and} \quad \bar{\Sigma}_{\R^{2}}(B((x,0),r)) \leq \bar{\Sigma}(B(x,Ar)) \end{displaymath}

Now, fix $x \in \R$ and $r \geq \Delta$ such that $\spt \bar{\Sigma} \not\subset B(x,A^{2}Dr)$ and $B(x,A^{2}Dr) \subset [-2,2]$. This implies 
$$\spt \bar{\Sigma}_{\R^{2}} \not\subset B((x,0),ADr), \text{ and } B((x,0),ADr) \subset B(0,3).$$ 
So, by the $(D,\beta)$-uniform perfectness of $\bar{\Sigma}_{\R^{2}}$ on $B(0,3)$,
\begin{displaymath} \bar{\Sigma}(B(x,r)) \leq \bar{\Sigma}_{\R^{2}}(B((x,0),Ar)) \leq \beta \cdot \bar{\Sigma}_{\R^{2}}(B((x,0),ADr)) \leq \beta \cdot \bar{\Sigma}(B(x,A^{2}Dr)).\end{displaymath}
This proves the $(A^{2}D,\beta,[-2,2])$-uniform perfectness of $\bar{\Sigma}$. By definition of $\bar{\sigma}$ we are done. \end{proof}

We next define the following set $G := G_{R,\theta} \subset \delta \Z \times \delta \Z$:
\begin{displaymath} G := (T_{1} \times T_{2})(\cup \mathcal{G}_{R,\theta}), \quad \text{where} \quad \mathcal{G}_{R,\theta} := \{(p,q) \in \mathcal{G} : p \in \mathcal{D}_{\delta}(\mathbf{Q} \cap R) \text{ and } q \in \mathcal{D}_{\delta}(\theta)\}. \end{displaymath}
\begin{claim}\label{claim1} It holds $(\bar{\nu}_{R} \times \bar{\sigma}|_{[0,1]})(G) \geq \delta^{4\epsilon}$, so in particular $\bar{\sigma}([0,1]) \geq \delta^{4\epsilon}$.
\end{claim}

\begin{proof} First, note that by definition
$$(\nu_{R} \times \sigma_{\theta})(\cup \mathcal{G}_{R,\theta}) = (\nu_{R} \times \sigma_{\theta})(\cup \mathcal{G}).$$ 
Also, recall from \nref{1} that 
$$ (\nu_{R} \times \sigma_{\theta})(\cup \mathcal{G}) \geq \delta^{4\epsilon}\|\nu_{R}\|\|\sigma_{A\theta}\|.$$ Define the auxiliary measure $\bar{\sigma}_{\theta} := \|\sigma_{A\theta}\|^{-1}T_{2}\sigma_{\theta}$. Then, it follows from Lemma \ref{lemma5} that
\begin{displaymath} (\bar{\nu}_{R} \times \bar{\sigma}_{\theta})(G) \stackrel{\mathrm{def.}}{=} \frac{(T_{1}\nu_{R} \times T_{2}\sigma_{\theta})[(T_{1} \times T_{2})(\cup \mathcal{G}_{R,\theta})]}{\|\nu_{R}\|\|\sigma_{A\theta}\|} \stackrel{\textup{L.\, \ref{lemma5}}}{\geq} \frac{(\nu_{R} \times \sigma_{\theta})(\cup \mathcal{G}_{R,\theta})}{\|\nu_{R}\|\|\sigma_{A\theta}\|} \stackrel{\textup{\nref{1}}}{\geq} \delta^{4\epsilon}. \end{displaymath}
To complete the proof, we claim that $\bar{\sigma}|_{[0,1]} \geq \bar{\sigma}_{\theta}$ in the sense of measures. This follows by noting that $\spt \bar{\sigma}_{\theta} \subset T_{2}(R(\theta)) \subset [0,1]$ by \eqref{form43}, so 
\begin{align*} \bar{\sigma}|_{[0,1]} & = ((T_{2}\sigma_{A\theta})[0,1])^{-1}(T_{2}\sigma_{A\theta})|_{[0,1]}\\
& \geq \|\sigma_{A\theta}\|^{-1}(T_{2}\sigma_{\theta})|_{[0,1]}\\
& = \|\sigma_{A\theta}\|^{-1}T_{2}\sigma_{\theta} = \bar{\sigma}_{\theta}. \end{align*}
This completes the proof of the claim. \end{proof}

For the next claim, write $X_{R} := \spt \bar{\nu}_{R} = T_{1}(\cup (R \cap \mathbf{Q} \cap \mathcal{X}))$. Recall that by remark \ref{rem4}, 
$$|X_{R}| \sim |R \cap \mathbf{Q} \cap \mathcal{X}|.$$

\begin{claim}\label{claim2} If $\delta > 0$ is small enough, $|\{x + y : (x,y) \in G\}| \leq \delta^{-12\epsilon}|X_{R}|$ \end{claim}

\begin{proof} Let $(x,y) \in G$. Thus, there exist squares $p \in \mathbf{Q} \cap R \cap \mathcal{X}$ and $q \in \mathcal{D}_{\delta}(\spt \sigma \cap \theta)$ such that $(p,q) \in \mathcal{G}_{R,\theta}$, and
\begin{displaymath} |(x,0) - T_{1}'(x_{0})| \lesssim \Delta \quad \text{and} \quad  |(y,0) - T_{2}'(y_{0})| \lesssim \Delta, \end{displaymath}
where $x_{0} \in p$ and $y_{0} \in q$ are arbitrary. Therefore
\begin{equation}\label{form30} |[(x,0) + (y,0)] - [T_{1}'(x_{0}) + T_{2}'(y_{0})]| \lesssim \Delta \end{equation}
Recall from \eqref{form29} that $T_{1}'(z) = \mathcal{O}((2\Delta)^{-1}z) + z_{1}$, and $T_{2}' = T_{1}' - z_{2}$. Consequently
\begin{displaymath} T_{1}'(x_{0}) + T_{2}'(y_{0}) = \mathcal{O}((2\Delta)^{-1}(x_{0} + y_{0})) + z_{1} - z_{2}. \end{displaymath} 
Combining this equation with \eqref{form30}, we se that every point in the set $\{(x,0) + (y,0) : (x,y) \in G\}$ is contained at distance $\lesssim \Delta$ from the set
\begin{equation}\label{form32} \{\mathcal{O}((2\Delta)^{-1}(p + q)) + z_{1} - z_{2} : (p,q) \in \mathcal{G}_{R,\theta}\} \end{equation} 
Finally, recall from \nref{2} that 
\begin{displaymath} |\{p + q : (p,q) \in \mathcal{G}_{R,\theta}\}|_{\delta} \leq \delta^{-11\epsilon}|R \cap \mathbf{Q} \cap \mathcal{X}| \sim \delta^{-11\epsilon}|X_{R}|. \end{displaymath}
Since $\Delta = \sqrt{\delta}$, the $\Delta$-covering number of the set in \eqref{form32} is $\lesssim \delta^{-11\epsilon}|X_{R}|$, and the claim now follows. \end{proof}
We are now in position to apply Lemma \ref{lemma2}: Condition (1) is met by Remark \ref{rem4}, and condition (2) is met by Claims \ref{claim1} and \ref{claim2}. We conclude that
\begin{equation} \label{eq: contradiction}
 \|\bar{\nu}_{R} \ast \bar{\sigma}|_{[0,1]}\|_{L^{2},\mathrm{Sh}}^{2} \gtrsim \delta^{28\epsilon}\|\bar{\nu}_{R}\|_{L^{2},\mathrm{Sh}}^{2} = \Delta^{56\epsilon}\|\bar{\nu}_{R}\|_{L^{2},\mathrm{Sh}}^{2}. 
\end{equation}
This places us a in position to apply Proposition \ref{RSProp}:  by Claim \ref{claim3} we know that $\bar{\sigma}$ is an $(A^{2}D,\beta,[-2,2])$-uniformly perfect with $\diam(\spt \bar{\sigma}) \gtrsim D^{-1}$, and $\bar{\sigma}([0,1]) \geq \delta^{4\epsilon}$ by Claim \ref{claim1}. Since $56\epsilon < \epsilon_{0}(A^{2}D,\beta,\eta)$ by our initial choice of parameters, the conclusion is that 
$$|X_{R}|^{-1} \sim \|\bar{\nu}_{R}\|_{L^{2},\mathrm{Sh}}^{2} \leq \Delta^{1 - \eta}.$$
Indeed, otherwise Proposition \ref{RSProp} would contradict \eqref{eq: contradiction}.

The equation above is equivalent to 
 $$|R \cap \mathbf{Q} \cap \mathcal{X}| \sim |X_{R}| \gtrsim \Delta^{\eta - 1}.$$
  This completes the proof of \eqref{form37}, and therefore the proof of Claim \ref{c1}. \end{proof} 

\subsubsection{Proof of Lemma \ref{lemma1}} \label{Section:end of proof of lemma1}
We may now complete the proof of Lemma \ref{lemma1}. Recall from \eqref{form12} that 
$$\sigma(\cup \{\theta \in \Theta : \mathbf{Q} \in \mathcal{G}_{\theta}\}) \geq \delta^{7\epsilon}.$$  Further, for each of those $\theta \in \Theta$ such that $\mathbf{Q} \in \mathcal{G}_{\theta}$, we infer from Claim \ref{c1} the existence of a certain set $\mathcal{X}_{\theta} \subset \mathcal{X} \cap \mathbf{Q}$ with $|\mathcal{X}_{\theta}| \geq \delta^{5\epsilon}|\mathcal{X} \cap \mathbf{Q}|$. Let $\boldsymbol{\sigma}$ be the discrete measure on the family $\Theta$ determined by $\boldsymbol{\sigma}(\theta) := \sigma(\theta)$. Then, by Cauchy-Schwarz,
\begin{align*} \sum_{\theta_{1} \in \Theta} \sum_{\theta_{2} \in \Theta} |\mathcal{X}_{\theta_{1}} \cap \mathcal{X}_{\theta_{2}}|\boldsymbol{\sigma}(\theta_{1})\boldsymbol{\sigma}(\theta_{2}) & = \sum_{p \in \mathcal{X} \cap \mathbf{Q}} \boldsymbol{\sigma}(\{\theta \in \Theta : p \in \mathcal{X}_{\theta}\})^{2}\\
& \geq |\mathcal{X} \cap \mathbf{Q}|^{-1} \Big( \sum_{p \in \mathcal{X} \cap \mathbf{Q}} \boldsymbol{\sigma}(\{\theta \in \Theta : p \in \mathcal{X}_{\theta}\}) \Big)^{2}\\
& = |\mathcal{X} \cap \mathbf{Q}|^{-1} \Big( \sum_{\theta \in \Theta} \sigma(\theta)|\mathcal{X}_{\theta}| \Big)^{2}\\
& \geq |\mathcal{X} \cap \mathbf{Q}|^{-1}\left( \delta^{12\epsilon}|\mathcal{X} \cap \mathbf{Q}| \right)^{2} = \delta^{24\epsilon}|\mathcal{X} \cap \mathbf{Q}|.\end{align*} 
Since $\sum_{\theta \in \Theta} \sigma(\theta) \lesssim 1$ by bounded overlap, the previous inequality implies the existence of $\theta_{1} \in \Theta$ such that
\begin{equation}\label{form40} \sum_{\theta_{2} \in \Theta} |\mathcal{X}_{\theta_{1}} \cap \mathcal{X}_{\theta_{2}}| \sigma(\theta_{2}) \gtrsim \delta^{24\epsilon}|\mathcal{X} \cap \mathbf{Q}|. \end{equation} 
To apply this information, recall that $\sigma$ is a $(D,\beta)$-uniformly perfect probability measure with $\diam(\spt \sigma) \geq \mathfrak{d}$ by hypothesis. Therefore, by Lemma \ref{frostmanLemma}, we have the following Frostman condition with exponent $s = -\log \beta/\log D > 0$:
\begin{displaymath} \sigma(B(x,r)) \leq (2D/\mathfrak{d})^{s} \cdot r^{s}, \qquad x \in \R^{2}, \, r > 0. \end{displaymath}

Using this and \eqref{form40}, we claim that, provided $\delta > 0$ is small enough in terms of $\epsilon,\mathfrak{d},D,s$, there exists $\theta_{2} \in \Theta$ satisfying 
$$\dist(\theta_{1},\theta_{2}) \geq \delta^{25\epsilon/s}, \text{ and } |\mathcal{X}_{\theta_{1}} \cap \mathcal{X}_{\theta_{2}}| \geq \delta^{25\epsilon}|\mathcal{X} \cap \mathbf{Q}|.$$
Indeed, otherwise
\begin{displaymath} \delta^{24\epsilon}|\mathcal{X} \cap \mathbf{Q}| \leq  |\mathcal{X} \cap \mathbf{Q}| \sum_{\dist(\theta_{1},\theta_{2}) \leq \delta^{25\epsilon/s}} \sigma(\theta_{2}) + \delta^{25\epsilon}|\mathcal{X} \cap \mathbf{Q}| \sum_{\theta_{2} \in \Theta} \sigma(\theta_{2}). \end{displaymath}
The second sum is $\lesssim \delta^{25\epsilon}|\mathcal{X} \cap \mathbf{Q}|$ by the bounded overlap of the sets $\theta_{2}$, and the first sum is also $\lesssim_{\mathfrak{d},D,s} \delta^{25\epsilon}|\mathcal{X} \cap \mathbf{Q}|$ by the Frostman condition. This leads to a contradiction, so the existence of $\theta_{2}$, as above, has been verified.

Write $\mathcal{X}_{Q} := \mathcal{X}_{\theta_{1}} \cap \mathcal{X}_{\theta_{2}}$, thus
\begin{equation}\label{form44} \max_{j \in \{1,2\}} |\pi_{\theta_{j}}(\cup \mathcal{X}_{Q})| \lesssim \Delta^{1 - \eta} |\mathcal{X} \cap \mathbf{Q}| \end{equation} 
according to Claim \ref{c1}. On the other hand, 
\begin{displaymath} \|\pi_{\theta_{1}} - \pi_{\theta_{2}}\| \gtrsim \dist(\theta_{1},\theta_{2}) \geq \delta^{25\epsilon/s}, \end{displaymath}
since the slope of $\pi_{\theta}$ is determined by $\varphi'(x_{\theta})$ (recall Remark \ref{rem6}), and $|\varphi'(x_{\theta_{1}}) - \varphi'(x_{\theta_{2}})| \sim \dist(\theta_{1},\theta_{2})$. We now record an elementary lemma on well-spaced orthogonal projections. For $e\in S^1$ let us write the corresponding orthogonal projection $\pi_{e}(x) := x \cdot e$.
\begin{lemma} \label{lemma:trans} 
 Let $e_{1},e_{2} \in S^{1}$, and write $\alpha := \|\pi_{e_{1}} - \pi_{e_{2}}\|$. Let $\delta \in (0,\tfrac{1}{2}]$, and let $\mathcal{Y} \subset \mathcal{D}_{\delta}([0,1)^{2})$. Then,
\begin{displaymath} \max_{j \in \{1,2\}} |\pi_{e_{j}}(\cup \mathcal{Y})|_{\delta} \gtrsim (\alpha |\mathcal{Y}|)^{1/2}. \end{displaymath} 
\end{lemma}

\begin{proof} For $j \in \{1,2\}$, let $\mathcal{T}_{j}$ be a family of $\delta$-tubes parallel to $\pi_{e_{j}}^{-1}$ such that 
\begin{displaymath} \cup \mathcal{Y} \subset \bigcup_{T \in \mathcal{T}_{j}} T \quad \text{and} \quad |\mathcal{T}_{j}| \sim |\pi_{e_{j}}(\cup \mathcal{Y})|_{\delta}. \end{displaymath}
Then $\cup \mathcal{Y} \subset \bigcup_{T_{1} \in \mathcal{T}_{1}} \bigcup_{T_{2} \in \mathcal{T}_{2}} T_{1} \cap T_{2}$. For $T_{1} \in \mathcal{T}_{1}$ and $T_{2} \in \mathcal{T}_{2}$ fixed, note (by elementary geometry) that $\diam(T_{1} \cap T_{2}) \lesssim \delta/\alpha$, so 
\begin{displaymath} |\{p \in \mathcal{Y} : p \cap T_{1} \cap T_{2} \neq \emptyset\}| \lesssim \alpha^{-1}. \end{displaymath}
Therefore $|\mathcal{Y}| \lesssim \alpha^{-1}|\mathcal{T}_{1}||\mathcal{T}_{2}|$, and the lemma follows. \end{proof}  
Applying the lemma to the projections $\pi_{\theta_{1}},\pi_{\theta_{2}}$ and $\mathcal{Y} := \mathcal{X}_{Q}$, we find
\begin{displaymath} \max_{j \in \{1,2\}} |\pi_{\theta_{j}}(\cup \mathcal{X}_{Q})| \gtrsim (\delta^{25\epsilon/s}|\mathcal{X}_{Q}|)^{1/2} \geq (\delta^{50\epsilon/s}|\mathcal{X} \cap \mathbf{Q}|)^{1/2}. \end{displaymath} 
Combining this estimate with \eqref{form44}, we deduce
\begin{displaymath} |\mathcal{X} \cap \mathbf{Q}| \gtrsim \delta^{50\epsilon/s}\Delta^{-2 + 2\eta} = \delta^{-1 + \eta + 50\epsilon/s}. \end{displaymath}
This contradicts the hypothesis \eqref{eq:sizeXloc} by our choice of "$\epsilon$" at \eqref{choiceEpsilon}.  The proof is complete.

\section{$L^2$-flattening and proof of Theorem \ref{thm:main}} 
In this section we complete the proof of Theorem \ref{thm:main}. With Proposition \ref{prop:main} in hand, the argument is adapted from the deduction of \cite[Theorem 1.1]{Orponen2024Jan} from \cite[Proposition 4.3]{Orponen2024Jan}, except that we work with $L^2$ norms rather than energies.

We first apply Proposition \ref{prop:main} to obtain an $L^2$-flattening statement of roughly the following form: if $\sigma$ is a $(D,\beta)$-uniformly perfect measure on $\mathbb{P}$ and $\mu$ is a probability measure on $\R^{2}$, then the $L^2$ norm of $(\mu \ast \sigma)^{k}_{\delta}$ is smaller than that of $\mu_{\delta}$ by a factor of $\delta^{\eta}$, provided $k \in \N$ is sufficiently large. The precise formulation is Proposition \ref{prop:4.7}. Applying this result iteratively shows that the $L^2$ norm of $\sigma_{\delta}^{k}$ can be made $\leq \delta^{-\epsilon}$ by choosing $k = k(\epsilon) \geq 1$ sufficiently large. The details are carried out in Corollary \ref{cor2}. From there, Theorem \ref{thm:main} follows easily, see Section \ref{s:finalProof}.

\subsection{$L^2$-flattening} We start with the following corollary of Proposition \ref{prop:main}. We assume the same setting as in Proposition \ref{prop:main}, except for allowing for $(\mu*\sigma)^k$ with $k>1$ in assumption (3); this is, however, as easy consequence of the case $k=1$. 

The Corollary is  more general than  Proposition \ref{prop:main}, since we also show that the amount of $\epsilon$ "gain" is   bounded away from zero  when keeping $\beta,D$  fixed, with $\alpha$ ranging on a compact subinterval of $(0,2)$.
We deduce this  \emph{a posteriori} by a compactness argument. Another possibility would be to track the dependence throughout the proof of the original Proposition \ref{prop:main}. While in principle straightforward, this would be a little tedious: eventually the dependence between $\epsilon$ and $\alpha$ is affected by the dependence between $\epsilon$ and $\zeta$ in Theorem \ref{thm:inverse}, and it has not been explicitly stated in \cite{Sh} that $\epsilon$ stays bounded away from zero when $\zeta$ does the same.

\begin{cor} \label{cor:convolutions}
For all $\alpha \in [0,2)$, $\beta \in [0,1)$, $D > 1$, $\mathfrak{d} > 0$ there exist $\epsilon = \epsilon(\alpha,\beta,D) > 0$ and $\delta_{0} = \delta_{0}(\alpha,\beta,\epsilon,\mathfrak{d},D) > 0$ such that the following holds for all $\delta \in (0,\delta_{0}]$. 

Assume that $k \in \N$,  $\mu, \sigma$ are Radon measures, and $E \subset \R^{2}$ is a Borel set such that:
\begin{enumerate}
\item $\mu$ is  supported on a dyadic cube of side length $1$, $\mu(\R^{2}) \leq 1$, and $\|\mu_{\delta}\|_{2}^{2} \leq \delta^{\alpha - 2 - \epsilon}$;

\item  $\sigma$ is  $(D,\beta)$-uniformly perfect, $\sigma(\R^{2}) \leq 1$, $\spt \sigma \subset \mathbb{P}$, and $\diam(\spt \sigma) \geq \mathfrak{d}$;

\item $(\mu \ast \sigma)^k(E) \geq \delta^{\epsilon}$.
\end{enumerate}
Then, 
$$|E|_{\delta} \geq \delta^{-\alpha - \epsilon}.$$
Moreover, the constant $\epsilon > 0$ stays bounded away from zero when $\beta,D$ are fixed, and $\alpha$ ranges on a compact subinterval of $(0,2)$. 
\end{cor}

\begin{proof}
First, by definition,
$$(\mu \ast \sigma)^k(E) = \int (\mu \ast \sigma)\left( E-z_2-z_3-\dots-z_k \right)\,d(\mu \ast \sigma)(z_1)\,\dots d (\mu \ast \sigma)(z_k).$$
Therefore, the assumption $(\mu \ast \sigma)^k(E) \geq \delta^{\epsilon}$ implies that
$$(\mu \ast \sigma)(E-z_2-z_3-\dots-z_k) \geq \delta^{\epsilon} \text{ for some } z_2,z_3,\dots,z_k\in \mathbb{R}^2.$$
Now, we use the set $E-z_2-z_3-\dots-z_k$ in assumption (3) of Proposition \ref{prop:main}. Applying the proposition and noting that $|E|_\delta = |E-z_2-z_3-\dots-z_k|_\delta$, the corollary follows.

Next, we prove the uniformity in $\epsilon$ as in the last assertion of the Corollary. As we have just seen, there is no loss of generality in assuming $k=1$, so we focus on this case. Fix $\beta,D$ as in Proposition \ref{prop:main}, and let $I \subset (0,2)$ be a compact interval, and fix $\bar{\alpha} \in (\max I,2)$. We have already proved that Proposition \ref{prop:main} is valid for each $\alpha \in I$. Let $\epsilon_{\alpha} := \epsilon(\alpha,\beta,D) > 0$ be the constant produced by the proposition. The open intervals $B(\alpha,\epsilon_{\alpha}/4), \alpha\in I,$ cover $I$, so by compactness we may choose a finite subset $\mathcal{A} \subset I$ such that  the intervals $\{B(\alpha,\epsilon_{\alpha}/4) : \alpha \in \mathcal{A}\}$ already cover $I$. Set
\begin{equation} \label{eq: choice of eps}
\epsilon := \epsilon(I,\beta,D) := \min\{\epsilon_{\alpha}/4 : \alpha \in \mathcal{A}\} > 0.
\end{equation} 
Now we claim that this "$\epsilon$" works simultaneously for all $\alpha \in I$. 

Let $\alpha \in I$, and let $\mu,\sigma,E$ be objects satisfying (1)-(3) with constants $(\alpha,\epsilon)$ (recall again that we are assuming for this part that $k=1$). We claim that $|E|_{\delta} \geq \delta^{-\alpha - \epsilon}$. To begin with, pick $\alpha' \in \mathcal{A}$ such that 
$$|\alpha - \alpha'| < \frac{\epsilon_{\alpha'}}{4}.$$
Let us check that $\mu,\sigma,E$ satisfy hypotheses (1)-(3) with constants $(\alpha',\epsilon_{\alpha'})$. Regarding (1), 
\begin{displaymath} \|\mu_{\delta}\|_{2}^{2} \leq \delta^{\alpha - 2 - \epsilon} \leq \delta^{\alpha' - 2 - \epsilon - |\alpha - \alpha'|} \leq \delta^{\alpha' - 2 - \epsilon_{\alpha'}}. \end{displaymath} 
Part (2) holds trivially, and part (3) follows from our choice of $\epsilon$ \eqref{eq: choice of eps}, as 
$$(\mu \ast \sigma)(E) \geq \delta^{\epsilon} \geq \delta^{\epsilon_{\alpha'}}.$$ Now that (1)-(3) have been verified for the pair $(\alpha',\epsilon_{\alpha'})$, we may finally draw the desired conclusion
\begin{displaymath} |E|_{\delta} \geq \delta^{-\alpha' - \epsilon_{\alpha'}} \geq \delta^{|\alpha - \alpha'|}\delta^{-\alpha - \epsilon_{\alpha'}} \geq \delta^{-\alpha - \epsilon_{\alpha'}/2} \geq \delta^{-\alpha - \epsilon}. \end{displaymath}
\end{proof}

Proposition \ref{prop:4.7} below is modeled on \cite[Proposition 4.7]{Orponen2024Jan}. The proof, however, differs in that we work with $L^2$ norms instead of energies, and in particular we invoke Corollary \ref{cor:convolutions} in place of \cite[Proposition 4.3]{Orponen2024Jan}. Our argument also avoids the parabolic rescaling step in \cite[Proposition 4.7]{Orponen2024Jan}; it turns out that this can be dispensed with at the cost of a little additional pigeonholing.

\begin{proposition} \label{prop:4.7}
For all $\alpha \in [0,2)$, $\beta \in [0,1)$, $D > 1$, $\mathfrak{d} > 0$, and $R>1$ there exist constants
\begin{displaymath} \begin{cases}  \eta = \eta(\alpha,\beta,D) > 0,\\
 					k_0 = k_0(\alpha,\beta,D) \in \N,\\
					 \delta_{0} = \delta_{0}(\alpha,\beta,\mathfrak{d},D,R) > 0, \end{cases} \end{displaymath}
such that the following holds for all $\delta \in (0,\delta_{0}]$. 

Let  $\mu, \sigma$ be Radon measures such that:
\begin{enumerate}
\item $\mu(\R^{2}) \leq 1$, $\spt \mu \subset [-R,R)^{2}$, and $\|\mu_{\delta}\|_{2}^{2} \leq \delta^{\alpha - 2}$;

\item  $\sigma$ is  $(D,\beta)$-uniformly perfect, $\sigma(\R^{2}) \leq 1$, $\spt \sigma \subset \mathbb{P}$, and $\diam(\spt \sigma) \geq \mathfrak{d}$.

\end{enumerate}
Then,
$$\big\|(\mu \ast \sigma)_{\delta}^{k} \big\|_{2}^{2} \leq \delta^{\alpha + \eta - 2}, \qquad k \geq k_0 .$$
Moreover, the constant $\eta > 0$ stays bounded away from zero when $\beta,D$ are fixed, and $\alpha$ ranges on a compact subinterval of $[0,2)$. \end{proposition}

\begin{proof} In the following the implicit constants in the "$\lesssim$" may depend on the parameters $D,\alpha,\beta,\kappa$. We start by defining the parameters $\eta,\epsilon,k_{0}$ whose existence is claimed. We set
\begin{equation} \label{eq:lambda, eps}
\eta := \min\{\tfrac{1}{2}\epsilon(\alpha,\beta,D),2 - \alpha\} \quad \text{and} \quad k_{0} := 2^{\ceil{20/\eta}},
\end{equation}
where $\epsilon(\alpha,\beta,D) > 0$ is the constant given by Corollary \ref{cor:convolutions}. We then set up some further notation. For $k\in \mathbb{N}$, we denote
$$\Pi^k := \Pi^{k}_{\delta} :=  \left( \mu\ast \sigma \right)^k \ast \psi_{\delta},\text{ and } J(k):= \big\| \Pi ^{2^k}  \big\|_2.$$
We record that the sequence $J(k)$ is non-increasing by Young's inequality (or Plancherel), writing $\left\lVert \nu  \right\rVert_1$ for both total variance and $L^1$-norm:

$$ J(k+1)= \big\| \Pi ^{2^k} \ast \Pi ^{2^k} \big\|_2 \leq  \big\| \Pi^{2^k} \big\|_1 \cdot  \big\| \Pi^{2^k} \big\|_2 =  \big\| \Pi ^{2^k} \big\|_2 = J(k), \quad k \geq 0.$$
In particular, to prove Proposition \ref{prop:4.7}, it suffices to find $k \leq \ceil{20/\eta}$ such that
\begin{equation}\label{form57} J(k) \leq \delta^{ \frac{\alpha+\eta-2}{2} }. \end{equation}

\begin{claim} \label{claim:(4.10)} There either exists $k \leq \lceil 20/\eta \rceil$ such that 
\begin{equation}\label{form58} \delta^{\eta/10} J (k) \leq J(k+1)\leq J(k), \end{equation}
or otherwise \eqref{form57} holds with $k = \ceil{20/\eta}$ (provided $\delta > 0$ is sufficiently small in terms of $R$).
\end{claim}
\begin{proof} Suppose that $\delta^{\eta/10} J (k) \leq J (k+1)$ for every $k\leq \lceil 20/\eta \rceil$. Applying this $ \lceil 20/\eta \rceil$ times,
$$J( \lceil 20/\eta \rceil  )\leq \delta^{(20/\eta) \cdot (\eta/10)} J(0)  \lesssim_{R} \delta^2 \cdot \delta^{-2} =1.$$
This implies \eqref{form57} for $\delta=\delta(R) > 0$ small enough. \end{proof}
For the remainder of the proof, we may assume that the first option in Claim \ref{claim:(4.10)} holds: there exists $k \leq \ceil{20/\eta}$ satisfying \eqref{form58}. We claim that \eqref{form57} is satisfied with this $k$.

We start by performing a discretisation at scale $\delta$ of the function $\Pi^{2^k}$. First, define
$$a_Q :=\sup_{x\in Q} \Pi ^{2k} (x),\qquad Q\in \mathcal{D}_{\delta}(\mathbb{R}^2). $$
Now, define 
$$A_0:= \bigcup \lbrace Q\in \mathcal{D}_{\delta} (\mathbb{R}^2):\, a_Q\leq 1 \rbrace,$$
and for $j\geq 1$,
\begin{equation}\label{decomposition} A_0:= \bigcup_{j} \lbrace Q\in \mathcal{D}_{\delta}(\mathbb{R}^2):\, 2^{j-1} \leq a_Q\leq 2^{j} \rbrace. =: \bigcup_{j} A_{j}. \end{equation}
Note that the sets $A_j$ are disjoint for distinct $j$, and $A_j=\emptyset$ for all $j \geq C\log( \frac{1}{\delta} )$ (for $C \geq 1$ absolute) since  $\big\|  \Pi ^{2^k} \big\| \lesssim \delta^{-2}$ for $k\geq 1$.

The following claim is \cite[Claim 4.13]{Orponen2024Jan}, but we repeat the details for completeness. 
\begin{claim}\label{claim:4.13}
There exists $j \in \{0,\ldots,C\log(1/\delta)\}$ such that writing $A := A_{j}$:
\begin{enumerate}
\item $\big\|  \Pi ^{2^k} \big\|_2 \lesssim_{\eta} |A|_{\delta} ^{-1/2}\delta^{-1-\eta/5}$;

\item $  \Pi_{8\delta} ^{2^k} (A) \gtrsim_{\eta} \delta^{\eta/5}.$
\end{enumerate}
\end{claim}

  \begin{proof} The next estimate follows from \eqref{decomposition}, and $a_{Q} \lesssim \inf_{x \in Q} \Pi_{8\delta}^{2^{k}}(x)$:
   \begin{equation}\label{eq-discretization2}
\Pi^{2^{k}} \leq \sum_{j=0}^{C\log (1/\delta)} 2^j \cdot \mathbf{1}_{A_j} \quad \text{and} \quad 2^{j} \cdot \mathbf{1}_{A_{j}} \lesssim C\Pi_{8\delta}^{2^{k}} \text{ for all } j \geq 0.
    \end{equation}
(See \cite[(4.11)-(4.12)]{Orponen2024Jan} for full details.) Taking $L^2$-norms, using $\psi_{r} \lesssim \psi_{r} \ast \psi_{r}$, and applying the triangle inequality, we may pigeonhole an index $j\geq 0$ and a set $A = A_j$ such that 
    $$
\Vert \Pi^{2^{k+1}}\Vert_2 \lesssim \Vert \Pi^{2^{k}} * \Pi^{2^k}\Vert_2 \lesssim \log (1/\delta) \cdot 2^j \Vert \mathbf{1}_A * \Pi^{2^k}\Vert_2.
    $$
    On the other hand, by Plancherel and Cauchy-Schwarz,
    $$
\Vert \mathbf{1}_A * \Pi_r^{2^k}\Vert_2 \leq \Vert \mathbf{1}_A * \mathbf{1}_A\Vert_2^{1/2} \Vert \Pi^{2^{k+1}}\Vert_2^{1/2}.
    $$
    Combining these two estimates with \eqref{form58} yields
    \begin{equation}\label{eq-boundforpi0}
        \delta^{\eta/10} \Vert \Pi^{2^k}\Vert_2 \leq  \Vert \Pi^{2^{k+1}}\Vert_2 \lesssim (\log (1/\delta))^2 2^{2j} \Vert \mathbf{1}_A * \mathbf{1}_A\Vert_2 \lesssim_{\eta} \delta^{-\eta/10} 2^{2j} \Vert \mathbf{1}_A\Vert_1 \Vert\mathbf{1}_A\Vert_2.
    \end{equation}
   Here, $A$ is a union of elements in $\mathcal{D}_{\delta}(\R^{2})$, so $\Vert \mathbf{1}_A\Vert_1 = \delta^2|A|_{\delta}$ and  $\Vert \mathbf{1}_A\Vert_2 = \delta |A|_{\delta}^{1/2}$. In particular, combining \eqref{eq-boundforpi0} with $2^j\Vert \mathbf{1}_A\Vert_1 \lesssim \Vert \Pi_{8\delta}^{2^k}\Vert_1 \leq 1$ (see \eqref{eq-discretization2}), we obtain
    \begin{equation}\label{eq-boundforpi}
\Vert \Pi^{2^k}\Vert_2 \lesssim_{\eta} \delta^{-\eta/5} 2^{2j}\Vert \mathbf{1}_A\Vert_1 \Vert \mathbf{1}_A\Vert_2 \lesssim \delta^{-\eta/5} \Vert \mathbf{1}_A\Vert_1^{-1}\Vert \mathbf{1}_A\Vert_2 \leq \delta^{-\eta/5-1}|A|_{\delta}^{-\frac{1}{2}},
    \end{equation}
     concluding the proof of Claim \ref{claim:4.13}(1).
     
    Moving towards proving part (2), we first observe that, by \eqref{eq-discretization2}, we have
    \begin{equation}\label{eq-lowerboundfor4r}
2^j \Vert \mathbf{1}_{A}\Vert_2 \lesssim \Vert \Pi_{8\delta}^{2^{k}}\Vert_2.
    \end{equation}
	The radial decrease of the approximate identity implies that $\psi_{8\delta}\lesssim \psi_{8\delta}\ast\psi_{\delta}$, so
    \begin{equation}\label{eq-maximalinequality}
        	\|\Pi^{2^k}_{8\delta}\|_2 \lesssim \|\Pi^{2^k}\ast \psi_{8\delta}\ast\psi_{\delta}\|_2\leq \|\Pi^{2^k}\ast \psi_{\delta}\|_2 \|\psi_{8\delta}\|_1=\|\Pi^{2^k}\|_2.
    \end{equation}
    Finally,
    \begin{displaymath} \|\Pi^{2^{k}}\|_{2} \stackrel{\eqref{eq-boundforpi}}{\lesssim_{\eta}} \delta^{-\eta/5}2^{2j}\|\mathbf{1}_{A}\|_{1}\|\mathbf{1}_{A}\|_{2} \stackrel{\eqref{eq-lowerboundfor4r}}{\lesssim} \delta^{-\eta/5}2^{j}\|\mathbf{1}_{A}\|_{1}\|\Pi^{2^{k}}_{8\delta}\|_{2} \stackrel{\eqref{eq-maximalinequality}}{\lesssim} \delta^{-\eta/5}2^{j}\|\mathbf{1}_{A}\|_{1}\|\Pi^{2^{k}}\|_{2}, \end{displaymath}
    so $\delta^{\eta/5} \lesssim_{\eta} 2^j \cdot \Vert \mathbf{1}_A\Vert_1 \lesssim \Pi_{8\delta}^{2^{k}}(A)$ (by \eqref{eq-discretization2}). We have proved Claim \ref{claim:4.13}(2) \end{proof}

By Claim \ref{claim:4.13}(2) and since $\spt \psi_r \subset B(r)$,
\begin{equation} \label{eq (4.19}
\Pi ^{2^k} \left( [A]_{8r} \right) \geq \Pi_{8\delta} ^{2^k} (A) \gtrsim \delta^{\eta/5}, \text{ where } [A]_{8r} \text{ is the } 8r-\text{neighbourhood of }A.
\end{equation}
We will now apply Corollary \ref{cor:convolutions} to the measure $\Pi=\mu\ast \sigma$. A small technicality is that the corollary requires $\spt \mu \subset B(1)$, whereas here $\spt \mu \subset B(R)$. Using \eqref{eq (4.19}, and pigeonholing, there is a restriction of $\mu$ to some unit square $[a,a + 1] \times [b,b + 1]$, denoted here $\nu$,  such that 
$$(\nu * \sigma)([A]_{8\delta}+z) \gtrsim_{\eta}  \delta^{\eta}/R^2,\qquad z\in \mathbb{R}^2.$$ 
So, recalling that $\eta = \epsilon/2$ with $\epsilon = \epsilon(\alpha,\beta,D)$, and taking $\delta = \delta(\eta,R) > 0$ small enough,
$$(\nu * \sigma)([A]_{8\delta}+z) \geq \delta^{\epsilon}.$$ 
We also note that $\|\nu_{\delta}\|_{L^{2}}^{2} \leq \|\mu_{\delta}\|_{L^{2}}^{2} \leq \delta^{\alpha - 2}$ by hypothesis. We may therefore apply Corollary \ref{cor:convolutions} to the measures $\nu,\sigma$ and the set $E:=[A]_{8\delta}+z$, and conclude that
$$|A|_{\delta} \geq \delta^{-\alpha-\epsilon}. $$
So, by Claim \ref{claim:4.13}(1), and again recalling $\epsilon = 2\eta$,
$$J(k) = \big\| \Pi ^{2^k} \big\|_2 \lesssim_{\eta} |A|_{\delta}^{-\frac{1}{2}}\delta^{-1-\eta/5} \leq \delta^{\frac{\alpha+\epsilon}{2}}\cdot \delta^{-1-\eta/5} \leq  \delta^{ \frac{\alpha + 3\eta/2 -2}{2}}.$$ 
In particular $J(k) \leq \delta^{(\alpha + \eta - 2)/2}$ for $\delta > 0$ small enough. \end{proof}

The next corollary follows by a straightforward iteration of Proposition \ref{prop:4.7}:

\begin{cor}\label{cor2} For all $D \geq 1$, $\mathfrak{d} > 0$, $\beta \in [0,1)$, $t \in [0,2)$, there exist $k_{0} = k_{0}(D,\beta,t) \in \N$ and $\delta_{0} = \delta_{0}(D,\mathfrak{d},\beta,t) > 0$ such that the following holds for all $\delta \in (0,\delta_{0}]$. Assume that $\sigma$ is a $(D,\beta)$-uniformly perfect probability measure supported on $\mathbb{P}$ with $\diam(\spt \sigma) \geq \mathfrak{d}$. Then,
\begin{displaymath} \|\sigma_{\delta}^{k}\|_{2}^{2} \leq \delta^{t - 2}, \qquad k \geq k_{0}. \end{displaymath} 
\end{cor} 

\begin{proof} Recall from Lemma \ref{frostmanLemma} that $\sigma$ is an $\alpha$-dimensional Frostman measure with $\alpha := -\log \beta/\log D > 0$, more precisely
\begin{displaymath} \sigma(B(x,r)) \leq (2D/\mathfrak{d})^{\alpha} \cdot r^{\alpha}, \qquad x \in \R^{2}, r > 0. \end{displaymath}
Let $\alpha_{0} := \tfrac{1}{2}\alpha$. The Frostman property easily implies $\|\sigma_{\delta}\|_{L^{2}}^{2} \lesssim_{D,\mathfrak{d},\beta} \delta^{\alpha - 2}$, thus
\begin{equation}\label{form51} \|\sigma_{\delta}\|_{2}^{2} \leq \delta^{\alpha_{0} - 2}, \qquad \delta \in (0,\delta_{0}], \end{equation}
for some $\delta_{0} = \delta_{0}(D,\mathfrak{d},\beta) > 0$. We proceed to define a sequence of exponents $\{\alpha_{j}\}_{j = 0}^{\infty}$, where $\alpha_{0} > 0$ was already defined above. Given $j \geq 0$, we inductively define 
\begin{displaymath} \alpha_{j + 1} := \alpha_{j} + \eta_{j} > 0, \end{displaymath}
where $\eta_{j} := \eta(\alpha_{j},\beta,D) > 0$ is the constant provided by Proposition \ref{prop:4.7}. Note that $\alpha_{j} \nearrow 2$ as $j \to \infty$, since $\eta(\alpha,\beta,D) > 0$ stays bounded away from zero as $\alpha$ ranges on any fixed compact subset of $[0,2)$. In particular, given $t \in [0,2)$ as in the statement of the corollary, there exists $j_{0} = j_{0}(D,\beta,t) \in \N$ such that $\alpha_{j_{0}} \geq t$.

We make the following claim, to be proved by induction. Fix $j \geq 0$. Then, there exist $k^{j} = k^{j}(D,\beta,j) \in \N$ and $\delta^{j} = \delta^{j}(D,\mathfrak{d},\beta,j) > 0$ such that 
\begin{equation}\label{form50} \|\sigma_{\delta}^{k_{j}}\|_{2}^{2} \leq \delta^{\alpha_{j} - 2}, \qquad \delta \in (0,\delta^{j}]. \end{equation}
Applying this with $j := j_{0}$, and setting $k_{0} := k^{j_{0}}$, proves Corollary \ref{cor2}.

Let us then prove the inductive claim. The case $j = 0$ follows from \eqref{form51}, with $k^{0} = 1$. Let us then assume that the claim has already been established for some $j \geq 0$. Apply Proposition \ref{prop:4.7} with parameters $D,\beta$, $\alpha := \alpha_{j}$, and radius 
\begin{displaymath} R_{j} = 2k^{j}, \end{displaymath}
where $k^{j}$ is the integer from the inductive hypothesis \eqref{form50}. The conclusion is the existence of $k_{0} = k_{0}(\alpha_{j},\beta,D) \in \N$, and $\delta_{0} := \delta_{0}(\alpha_{j},\beta,\mathfrak{d},D,R_{j}) > 0$ such that the following holds for $\delta \in (0,\delta_{0}]$. If $\mu$ is a Borel probability measure satisfying
\begin{equation}\label{form52} \spt \mu \subset [-R_{j},R_{j}]^{2} \quad \text{and} \quad \|\mu_{\delta}\|_{L^{2}}^{2} \leq \delta^{\alpha_{j} - 2}, \end{equation}
then 
\begin{equation}\label{form53} \|(\mu \ast \sigma)_{\delta}^{k_{0}}\|_{2}^{2} \leq \delta^{\alpha_{j} + \eta - 2} = \delta^{\alpha_{j + 1} - 2}. \end{equation}
But now by the inductive hypothesis \eqref{form50}, and since $\spt \sigma \subset [-2,2]^{2}$, the measure $\mu = \sigma^{k_{j}}$ satisfies \eqref{form52} for $\delta \in (0,\delta^{j}]$. Therefore \eqref{form53} yields 
\begin{displaymath} \|\sigma_{\delta}^{k_{0}(k_{j} + 1)}\|_{2}^{2} \leq \delta^{\alpha_{j + 1} - 2}, \qquad \delta \in (0,\min\{\delta_{0},\delta^{j}\}]. \end{displaymath}
This gives \eqref{form50} with $k_{j + 1} := k_{0}(k_{j} + 1)$, and completes the proof. \end{proof} 

\subsection{Proof of Theorem \ref{thm:main}}\label{s:finalProof} We are in a position to prove Theorem \ref{thm:main}, repeated below:

\begin{thm}\label{thm:main2} For every $D \geq 1$, $\mathfrak{d} > 0$, $\beta \in (0,1]$, and $\epsilon \in (0,1)$ there exists $p=p(D,\beta,\epsilon)\geq 1$ such that the following holds. Let $\sigma$ be a $(D,\beta)$-uniformly perfect probability measure with $\spt \sigma \subset \mathbb{P}$ and $\diam(\spt \sigma) \geq \mathfrak{d}$. Then,
\begin{displaymath}
  \left\lVert \hat{\sigma}  \right\rVert_{L^p (B(R))}^{p} \lesssim_{D,\mathfrak{d},\beta,\epsilon}  R^{\epsilon},\quad R\geq 1.
\end{displaymath}
\end{thm}

\begin{proof}[Proof of Theorem \ref{thm:main2}] Fix $R \geq 1$, and let $\{\varphi_{\delta}\}_{\delta > 0}$ be an radially decreasing approximate identity with the property $|\widehat{\varphi_{\delta}}(\xi)| \gtrsim 1$ for $|\xi| \leq \delta$. Then, writing $\delta := R^{-1}$, one has
\begin{equation}\label{form54} \|\hat{\sigma}\|_{L^{p}(B(R))}^{p} \lesssim \|\sigma^{k}_{\delta}\|_{2}^{2}, \qquad p = 2k \in 2\N. \end{equation}
Write $t := 2 - \epsilon$, where $\epsilon \in (0,1)$ is the parameter from the statement. Then, according to Corollary \ref{cor2}, we have $\|\sigma_{\delta}^{k}\|_{2}^{2} \leq \delta^{-\epsilon}$, provided $k \geq k_{0}(D,\beta,\epsilon)$, and $0 < \delta \leq \delta_{0}(D,\mathfrak{d},\beta,\epsilon)$. This completes the proof in combination with \eqref{form54}. \end{proof} 

\bibliographystyle{plain}
\bibliography{references}

\end{document}